\documentclass{elsarticle}

\usepackage{etoolbox}
\makeatletter
\patchcmd{\ps@pprintTitle}{\footnotesize\itshape
       Preprint submitted to \ifx\@journal\@empty Elsevier
       \else\@journal\fi\hfill\today}{\relax}{}{}
\makeatother

\usepackage{amsmath}
\usepackage{color} 
\usepackage{amssymb}
\usepackage{amsthm}
\usepackage{MnSymbol}
\usepackage{pgf,tikz}
\usepackage{caption}

\usepackage{chngcntr}
\usepackage{lipsum}
\usepackage{color}
\usetikzlibrary{arrows}

\theoremstyle{plain}
\newtheorem{thm}{Theorem}[section]
\newtheorem{lem}{Lemma}[section]
\newtheorem{cor}{Corollary}[section]
\newtheorem{rem}{Remark}[section]
\newtheorem{prop}{Proposition}[section]

\newtheorem{defn}{Definition}[section]
\newtheorem{Fact}{Fact}[section]

\newtheorem{Example}{Example}[section]
\newtheorem{Notation}{Notation}[section]

\newtheorem*{Discussion}{Discussion}
\newtheorem{Observation}{Observation}[section]
\numberwithin{equation}{section}

\begin{document}
\nocite{*}
\sloppy

\newtheorem*{Convention}{Convention}

\begin{frontmatter}

\title{Decomposability and co-modular indices of tournaments}
\author{Houmem Belkhechine\fnref{}}
\address{University of Carthage, Bizerte Preparatory Engineering Institute, Bizerte, Tunisia}
\ead{houmem.belkhechine@ipeib.rnu.tn}

\author{Cherifa Ben Salha\fnref{}}
\address{University of Carthage, Faculty of Sciences of Bizerte, Bizerte, Tunisia}
\ead{cherifa.bensalha@fsb.u-carthage.tn}

\begin{abstract} 
Given a tournament $T$, a module of $T$ is a subset $X$ of $V(T)$ such that for $x, y\in X$ and $v\in V(T)\setminus X$, $(x,v)\in A(T)$ if and only if $(y,v)\in A(T)$. The trivial modules of $T$ are $\emptyset$, $\{u\}$ $(u\in V(T))$ and $V(T)$. The tournament $T$ is indecomposable if all its modules are trivial; otherwise it is decomposable. The decomposability index of $T$, denoted by $\delta(T)$, is the smallest number of arcs of $T$ that must be reversed to make $T$ indecomposable. The first author conjectured that for $n \geq 5$, we have  $\delta(n) = \left\lceil \frac{n+1}{4} \right\rceil$, where $\delta(n)$ is the maximum of $\delta(T)$ over the tournaments $T$ with $n$ vertices. In this paper we prove this conjecture by introducing the co-modular index of a tournament $T$, denoted by $\Delta(T)$, as the largest number of disjoint co-modules of $T$, where a co-module of $T$ is a subset $M$ of $V(T)$ such that $M$ or $V(T) \setminus M$ is a nontrivial module of $T$. We prove that for $n \geq 3$, we have $\Delta(n) = \left\lceil \frac{n+1}{2} \right\rceil$, where $\Delta(n)$ is the maximum of $\Delta(T)$ over the tournaments $T$ with $n$ vertices. Our main result is the following close relationship between the above two indices: for every tournament $T$ with at least $5$ vertices, we have $\delta(T) = \left\lceil \frac{\Delta(T)}{2} \right\rceil$. As a consequence, we obtain $\delta(n) = \left\lceil \frac{\Delta(n)}{2} \right\rceil = \left\lceil \frac{n+1}{4} \right\rceil$ for $n \geq 5$, and we answer some further related questions.            
\end{abstract}
\end{frontmatter}

\section{Introduction and presentation of results}
A {\it tournament} $T$ consists of a finite set $V(T)$ of {\it vertices} together with a set $A(T)$ of ordered pairs of distinct vertices, called {\it arcs}, such that for all $x \neq y \in V(T)$, $(x,y) \in A(T)$ if and only if $(y,x) \not\in A(T)$. Such a tournament is denoted by (V(T), A(T)). The {\it cardinality} of $T$, denoted by $v(T)$, is that of $V(T)$. Given a tournament $T$, with each subset $X$ of $V(T)$ is associated the {\it subtournament} $T[X] = (X, A(T) \cap (X \times X))$ of $T$ induced by $X$. For $X\subseteq V(T)$ (resp. $x\in V(T)$), the subtournament $T[V(T) \setminus X]$ (resp. $T[V(T) \setminus \{x\}$]) is simply denoted by $T-X$ (resp. $T-x$). Two tournaments $T$ and $T'$ are {\it isomorphic}, which is denoted by $T \simeq T'$, if there exists an {\it isomorphism} from $T$ onto $T'$, that is, a bijection $f$ from $V(T)$ onto $V(T')$ such that for every $x, y \in V(T)$, $(x, y) \in A(T)$ if and only if $(f(x), f(y)) \in A(T')$. With each tournament $T$ is associated its {\it dual} tournament $T^{\star}$ defined by $V(T^{\star})= V(T)$ and $A(T^{\star}) =\{(x,y) \colon\ (y,x)\in A(T)\}$. 

A {\it transitive} tournament is a tournament $T$ such that for every $x,y,z \in V(T)$, if $(x,y) \in A(T)$ and  $(y,z) \in A(T)$, then $(x,z) \in A(T)$. Let $n$ be a positive integer. We denote by $\underline{n}$ the transitive tournament whose vertex set is $\{0, \ldots, n-1\}$ and whose arcs are the ordered pairs $(i,j)$ such that $0 \leq i < j \leq n-1$. Up to isomorphism, $\underline{n}$ is the unique transitive tournament with $n$ vertices.

The paper is based on the following notions. Given a tournament $T$, a subset $M$ of $V(T)$ is a {\it module} \cite{Spinrad} (or a {\it clan} \cite{E} or an {\it interval} \cite{I}) of $T$ provided that for every $x,y \in M$ and for every $v \in V(T) \setminus M$, $(v,x) \in A(T)$ if and only if $(v,y) \in A(T)$. For example, $\emptyset$, $\{x\}$, where $x \in V(T)$, and $V(T)$ are modules of $T$, called {\it trivial} modules. 
A tournament is {\it indecomposable} \cite{I, ST} (or {\it prime} \cite{Spinrad} or {\it primitive} \cite{E}) if all its modules are trivial; otherwise it is {\it decomposable}. Let us consider some examples. To begin, consider the case of small tournaments. The tournaments with at most two vertices are clearly indecomposable. The tournaments $\underline{3}$ and $C_{3} = (\{0,1,2\}, \{(0,1), (1,2), (2,0)\})$ are, up to isomorphism, the unique tournaments with three vertices. The tournament $C_{3}$ is indecomposable, whereas $\underline{3}$ is decomposable. Up to isomorphism, the tournaments with four vertices are the four tournaments $\underline{4}$, $T_4=(\{0,1,2,3\}, \{(0,1), (1,2), (2,0), (3,0), (3,1), (3,2)\})$, $T_4^{\star}$, and $(\{0,1,2,3\}, \{(0,1), (1,2), (2,0), (3,0), (3,1), (2,3)\})$, all of them are decomposable. We now consider the case of transitive tournaments. 
For every integer $n \geq 3$, the transitive tournament $\underline{n}$ is decomposable. More precisely,
\begin{equation} \label{eq mod inter n}
 \text{the modules of} \ \underline{n} \ \text{are the intervals of the usual total order on} \ V(\underline{n}).  
\end{equation}
Lastly, consider the cases of isomorphic tournaments and dual tournaments. Let $T$ and $T'$ be isomorphic tournaments. If $f$ is an isomorphism from $T$ onto $T'$, then a subset $M$ of $V(T)$ is a module of $T$ if and only if $f(M)$ is a module of $T'$. In particular, $T$ is indecomposable if and only if $T'$ is. 
Similarly, a tournament $T$ and its dual share the same modules. In particular, $T$ is indecomposable if and only if $T^{\star}$ is. These remarks justify that in certain proofs, tournaments are considered up to isomorphism and/or duality.

Let $T$ be a tournament. An {\it inversion} of an arc $a = (x,y) \in A(T)$ consists of replacing the arc $a$ by $a^{\star}$ in $A(T)$, where $a^{\star} = (y,x)$. The tournament obtained from $T$ after reversing the arc $a$ is denoted by ${\rm Inv}(T,a)$ or ${\rm Inv}(T,\{x,y\})$. Thus ${\rm Inv}(T,a) = {\rm Inv}(T,\{x,y\}) = (V(T), (A(T) \setminus \{a\}) \cup \{a^{\star}\})$. More generally, for $B \subseteq A(T)$, we denote by ${\rm Inv}(T, B)$ the tournament obtained from $T$ after reversing all the arcs of $B$, that is ${\rm Inv}(T, B)= (V(T), (A(T) \setminus B) \cup B^{\star})$, where $B^{\star} = \{b^{\star} \colon\ b \in B\}$. For example, $T^{\star}={\rm Inv}(T, A(T))$.

Given a tournament $T$ with at least five vertices,
the {\it decomposability index} of $T$, denoted by $\delta(T)$, was defined by the first author \cite{Index} as the least integer $m$ for which there exists $B \subseteq A(T)$ such that $|B|=m$ and ${\rm Inv}(T, B)$ is indecomposable. The index $\delta(T)$ is well-defined because, as observed in \cite{Index}, for every integer $n \geq 5$, there exist indecomposable tournaments with $n$ vertices. Notice that $\delta(T) = \delta(T^{\star})$.  
Similarly, isomorphic tournaments have the same decomposability index. The exact value of the decomposability index of transitive tournaments was found in \cite{Index}.

\begin{prop} [\cite{Index}] \label{deltatr}
 Given a transitive tournament $T_n$ with $n$ vertices, where $n \geq 5$, we have $\delta(T_n) = \left\lceil \frac{n+1}{4} \right\rceil$. 
\end{prop}

For $n \geq 5$, let $\delta(n)$ be the maximum of $\delta(T)$ over the tournaments $T$ with $n$ vertices. The first author \cite{Index} conjectured that $\delta(n) = \left\lceil \frac{n+1}{4} \right\rceil$ and asked some related questions. The original purpose of the paper is to prove or disprove this conjecture. We prove that this conjecture holds by establishing related results involving a new index, called co-modular index. As a consequence, we obtain the following theorem as well as answers for some further questions asked in \cite{Index}.

\begin{thm} \label{deltan}
For every integer $n \geq 5$, we have $\delta(n) = \left\lceil \frac{n+1}{4} \right\rceil$.
\end{thm}

\begin{Convention} \normalfont 
 Given a tournament $T$, for $X \subseteq V(T)$, $V(T) \setminus X$ is denoted by $\overline{X}$.
\end{Convention}

Given a tournament $T$, a {\it co-module} of $T$ is a subset $M$ of $V(T)$ such that $M$ or $\overline{M}$ is a nontrivial module of $T$. For instance,
\begin{equation} \label{eq comod dec}
 \text{a tournament} \ T \ \text{is decomposable if and only if} \ T \ \text{admits a co-module}.
\end{equation}
Observe that 
\begin{equation} \label{eq comod non vide}
\text{neither} \ \varnothing \ \text{nor} \ V(T) \ \text{is a co-module of} \ \text{some tournament} \ T.
\end{equation} 
Moreover, given a tournament $T$, contrary to the set of modules of $T$, 
\begin{equation} \label{eq close}
 \text{the set of co-modules of} \ T \ \text{is closed under complementation}.
\end{equation}
Given a tournament $T$, a {\it co-modular decomposition} of $T$ is a set of pairwise disjoint co-modules of $T$.
For instance, a tournament $T$ is decomposable if and only if it admits a nonempty co-modular decomposition. The {\it co-modular index} of a tournament $T$, denoted by $\Delta(T)$, is the largest size of a co-modular decomposition of $T$, i.e., the maximum of $\{|D| : D \ \text{is a co-modular decomposition of} \ T\}$.

 Contrary to the co-modular index, which is defined for all tournaments, the decomposability index is not defined for the tournaments with four vertices because, as observed above, these tournaments are all decomposable. Therefore, since the few tournaments with at most three vertices can easily be checked separately, we only consider tournaments with at least five vertices when the decomposability index is considered. 
For instance, given a tournament $T$ such that $v(T) \geq 5$, 
\begin{equation} \label{eq indec delta}
T \ \text{is  indecomposable  if  and only  if} \ \Delta(T) = \delta(T) = 0. 
\end{equation}
Notice that $\Delta (T)$ is never equal to $1$. Indeed, by (\ref{eq comod dec}) and (\ref{eq close}), 
\begin{equation} \label{eq dec delta}
\text{a  tournament T is decomposable if and only if} \ \Delta(T) \geq 2.
\end{equation}

Given a tournament $T$, a {\it $\Delta$-decomposition} of $T$ is a co-modular decomposition $D$ of $T$ which is of maximum size, i.e., such that $|D| = \Delta(T)$.  
 As observed for the decomposability index, for every isomorphic tournaments $T$ and $T'$, we have $\Delta(T) = \Delta (T') = \Delta(T^{\star})$. The next result is the analogue of Proposition~\ref{deltatr} for co-modular index.

\begin{prop} \label{Deltatr}
 Given a transitive tournament $T_n$ with $n$ vertices, where $n \geq 3$, we have $\Delta(T_n) = \left\lceil \frac{n+1}{2} \right\rceil$.
\end{prop}

\begin{proof}
 Up to isomorphism, we may assume  $T_n = \underline{n}$. Let us consider the co-modular decomposition $D_n$ of $T_n$ defined as follows (see (\ref{eq mod inter n})).
\begin{equation*} \label{eq D_n}
 D_n = \{\{0\}, \{n-1\}\} \cup \{\{2i-1, 2i\} \colon\ 1 \leq i \leq \lfloor \frac{n-2}{2}  \rfloor\}.
\end{equation*}
Clearly $|D_n| = \left\lceil \frac{n+1}{2} \right\rceil$ and $D_n$ is a co-modular decomposition of $T_n$. Thus $\Delta(T_n) \geq \left\lceil \frac{n+1}{2} \right\rceil$. Let $D'_n$ be another co-modular decomposition of $T_n$. Since $0$ and $n-1$ are the unique vertices $x$ of $T_n$ such that $\{x\}$ is a co-module of $T_n$, then $D'_n$ contains at most two singletons. Therefore, for every $M \in D'_n \setminus \{\{0\}, \{n-1\}\}$, we have $|M| \geq 2$ (see (\ref{eq comod non vide})). It follows that $n \geq 2|D'_n| -2$, i.e., $|D'_n| \leq \frac{n+2}{2}$. Thus, $|D'_n| \leq \left\lfloor \frac{n+2}{2} \right\rfloor = \left\lceil \frac{n+1}{2} \right\rceil$ so that $\Delta(T_n) \leq \left\lceil \frac{n+1}{2} \right\rceil$. We conclude that $\Delta(T_n) = \left\lceil \frac{n+1}{2} \right\rceil$.    
\end{proof}

Now, for $n \geq 3$, let $\Delta(n)$ be the maximum of $\Delta(T)$ over the tournaments $T$ with $n$ vertices.
The analogue of Theorem~\ref{deltan} for co-modular index (see Theorem~\ref{Deltan}) is a consequence of Proposition~\ref{Deltatr} and the following theorem due to Erd\H{o}s et al. \cite{EFHM}. 

\begin{Notation} \normalfont
 Given a tournament $T$, the set of the modules of $T$ is denoted by $\mathcal{M}(T)$. 
\end{Notation}

\begin{thm} [\cite{EFHM}] \label{erdos}
 Given a non-transitive tournament $T$, there exists a transitive tournament $T'$ such that $V(T') = V(T)$ and $\mathcal{M}(T) \varsubsetneq \mathcal{M}(T')$.
\end{thm}

\begin{thm} \label{Deltan}
 For every integer $n \geq 3$, we have $\Delta(n) = \left\lceil \frac{n+1}{2} \right\rceil$. 
\end{thm}

\begin{proof}
 Let $n \geq 3$. By Proposition~\ref{Deltatr}, it suffices to prove that $\Delta(n) \leq \left\lceil \frac{n+1}{2} \right\rceil$. Let $T$ be a tournament with $n$ vertices. By Theorem~\ref{erdos}, there exists a transitive tournament $T_n$ such that $V(T_n) = V(T)$ and $\mathcal{M}(T) \subseteq \mathcal{M}(T_n)$. Thus, every co-modular decomposition of $T$ is also a co-modular decomposition of $T_n$. It follows that $\Delta(T) \leq \Delta(T_n)$.  
 Since $\Delta(T_n) = \left\lceil \frac{n+1}{2} \right\rceil$ by Proposition~\ref{Deltatr}, we obtain $\Delta(n) \leq \left\lceil \frac{n+1}{2} \right\rceil$ as desired.         
\end{proof}

We will now see how the co-modular index is closely related to the decomposability one.

\begin{Notation} \normalfont
 Let $T$ be a tournament. For an arc $a = (x,y) \in A(T)$, the vertex set $\{x,y\}$ is denoted by $\mathcal{V}(a)$. Similarly, for an arc set $B \subseteq A(T)$, the vertex set $\displaystyle\bigcup_{b \in B} \mathcal{V}(b)$ is denoted by $\mathcal{V}(B)$. 
\end{Notation}

Let $T$ be a (decomposable) tournament. Let $B \subseteq A(T)$ and let $D$ be a co-modular decomposition of $T$. If $\mathcal{V}(B) \cap M = \varnothing$ for some $M \in D$, then $M$ is still a co-module of ${\rm Inv}(T, B)$, and in particular ${\rm Inv}(T, B)$ is still decomposable. Therefore, if ${\rm Inv}(T, B)$ is indecomposable, then $\mathcal{V}(B) \cap M \neq \varnothing$ for every $M \in D$, so that $|\mathcal{V}(B)| \geq |D|$ and thus $|B| \geq \frac{|\mathcal{V}(B)|}{2} \geq \frac{|D|}{2}$. It follows that when $v(T) \geq 5$, we have $\delta(T) \geq \frac{|D|}{2}$ and thus $\delta(T) \geq \left\lceil \frac{|D|}{2} \right\rceil$. We have shown that 
\begin{equation} \label{eq deltaDelta}
\delta(T) \geq \left\lceil \frac{\Delta(T)}{2} \right\rceil \ \text{for every tournament $T$ with at least five vertices.} 
\end{equation}
The most important result of the paper is certainly that equality holds in (\ref{eq deltaDelta}).

\begin{thm} \label{main}
 For every tournament $T$ with at least five vertices, we have $\delta(T) = \left\lceil \frac{\Delta(T)}{2} \right\rceil$. 
\end{thm}

We will now see how the main problems posed in \cite{Index} follow from Theorem~\ref{main}.
We begin by Theorem~\ref{deltan}, which is an immediate consequence of Theorems~\ref{main} and \ref{Deltan}.

\begin{proof}[Proof of Theorem \ref{deltan}.]
 Let $n \geq 5$. By Theorem~\ref{main}, $\delta(n) = \left\lceil \frac{\Delta(n)}{2} \right\rceil$. By Theorem~\ref{Deltan}, $\Delta(n) = \left\lceil \frac{n+1}{2} \right\rceil$. It follows that $\delta(n) = \left\lceil \frac{\left\lceil \frac{n+1}{2} \right\rceil}{2} \right\rceil = \left\lceil \frac{n+1}{4} \right\rceil$.  
\end{proof}

The next result, also conjectured in \cite{Index}, is another consequence of Theorem~\ref{main}.

\begin{cor} \label{modinclu}
 Given two tournaments $T$ and $T'$ such that $V(T) = V(T')$ and $v(T) \geq 5$, if $\mathcal{M}(T) \subseteq  \mathcal{M}(T')$, then $\delta(T) \leq \delta(T')$. 
\end{cor}

\begin{proof}
 Suppose $\mathcal{M}(T) \subseteq  \mathcal{M}(T')$. Since $V(T) = V(T')$ and $\mathcal{M}(T) \subseteq  \mathcal{M}(T')$, every co-modular decomposition of $T$ is also a co-modular decomposition of $T'$. It follows that $\Delta(T) \leq \Delta(T')$. Thus $\delta(T) \leq \delta(T')$ by Theorem~\ref{main}. 
\end{proof}

Notice that Theorem~\ref{deltan} is also an immediate consequence of Corollary~\ref{modinclu}, Theorem~\ref{erdos}, and Proposition~\ref{deltatr}.

Another application of Theorem~\ref{main} is about upward hereditary properties of the decomposability index based on the following question asked in \cite{Index}. 
For which values of the positive integer $k$ does the following property ($P_{k}$) hold?
\begin{center} \label{P_k}
($P_{k}$) \ For every tournament $T$ such that $v(T) \geq 5+k$, there exists a subset $X$ of $V(T)$ such that $|X|=k$ and $\delta(T) \leq \delta(T-X) +1$. 
\end{center}
As observed in \cite{Index}, Property ($P_{k}$) is false for $k \geq 5$. However, the trueness of Property ($P_{k}$) has been proved for $k \in \{1,2,3\}$, while Property ($P_{4}$) has been conjectured because it implies Theorem~\ref{deltan} (see \cite{Index}). Property ($P_{4}$) is a consequence of Theorem~\ref{main}. In fact, for each $k \in \{1,2,3,4\}$, Property ($P_{k}$) is a consequence of Theorem~\ref{main}. 
More precisely, let us consider the analogues ($Q_{k}$) of Properties ($P_{k}$) for the co-modular index:  
for which values of the positive integer $k$ does the following property ($Q_{k}$) hold?
\begin{center}
($Q_{k}$) \ For every tournament $T$ such that $v(T) \geq 3+k$, there exists a subset $X$ of $V(T)$ such that $|X|=k$ and $\Delta(T) \leq \Delta(T-X) +2$. 
\end{center}
By Theorem~\ref{main}, ($Q_{k}$) implies ($P_{k}$). Since ($P_{k}$) is false for $k \geq 5$, ($Q_{k}$) is also false for $k \geq 5$. By using Theorem~\ref{main}, we prove that for every $k \in \{1,2,3,4\}$, Property ($Q_k$), and thus Property ($P_k$), holds. We obtain the following theorem.

\begin{thm} \label{thmP}
 For every integer $k \in \{1,2,3,4\}$, the following two assertions are satisfied.
 \begin{enumerate}
 \item For every tournament $T$ such that $v(T) \geq 3+k$, there exists a subset $X$ of $V(T)$ such that $|X| = k$ and $\Delta(T) \leq \Delta(T-X) +2$.
 \item For every tournament $T$ such that $v(T) \geq 5+k$, there exists a subset $X$ of $V(T)$ such that $|X| = k$ and $\delta(T) \leq \delta(T-X) +1$.
 \end{enumerate}
\end{thm} 

\begin{proof}
 Let $k \in \{1,2,3,4\}$. As observed above, by  Theorem~\ref{main}, the first assertion implies the second one. Therefore, we only have to prove the first assertion. Assertion~1 clearly holds when $\Delta(T) \leq 2$. Let $T$ be a tournament such that $v(T) \geq 3+k$. 
 
To begin, suppose $\Delta(T) =3$ or $4$. Since $T$ is decomposable, $T$ admits a nontrivial module $M$. Let $x,y$ be distinct elements of $M$, let $z$ be an element of $\overline{M}$, and let $X$ be a subset of $V(T) \setminus \{x,y,z\}$ such that $|X|=k$. Since $M \setminus X = M \cap \overline{X}$ is a nontrivial module of $T[\overline{X}]=T-X$ (see Assertion~1 of Proposition~\ref{propo modules}), the tournament $T-X$ is decomposable, i.e., $\Delta(T-X) \geq 2$ (see (\ref{eq dec delta})). Thus $\Delta(T) \leq \Delta(T-X) +2$ as desired. 
 
 To finish, suppose $\Delta(T) \geq 5$. Let $D$ be a $\Delta$-decomposition of $T$. There exist two distinct elements $M$ and $N$ of $D$ such that $|M| \geq 2$ and $|N| \geq 2$ (see (\ref{eq comod non vide}) and Assertion~1 of Lemma~\ref{comod part}). Let $X$ be a subset of $M \cup N$ such that $|X| = k$. Since $|D| = \Delta(T) \geq 5$, the elements of $D \setminus \{M,N\}$ are co-modules of $T-X$ (see Assertion~1 of Proposition~\ref{propo modules}). Thus, $D \setminus \{M,N\}$ is a co-modular decomposition of $T-X$.  Since $|D \setminus \{M,N\}| = \Delta(T)-2$ because $|D| = \Delta(T)$, it follows that $\Delta(T) \leq \Delta(T-X)+2$, which completes the proof.    
\end{proof}

We end this section by showing how Theorem~\ref{main} results from the following propositions.

\begin{prop} \label{prop Delta 2}
 Given a tournament $T$ with at least five vertices,  $\delta(T) =1$ if and only if $\Delta(T) = 2$.
\end{prop}

\begin{prop} \label{prop Delta 3}
 Given a tournament $T$ with at least five vertices, if $\Delta(T) = 3$, then $\delta(T) = 2$.
\end{prop}

\begin{prop} \label{prop Delta 4}
 Given a tournament $T$ such that $\Delta(T) \geq 4$, there exists an arc $a \in A(T)$ such that $\Delta({\rm Inv}(T,a)) = \Delta(T) - 2$.
\end{prop}

\begin{proof}[Proof of Theorem \ref{main}.]
Let $T$ be a tournament such that $v(T) \geq 5$. We proceed by induction on $\Delta(T)$. By (\ref{eq indec delta}), (\ref{eq dec delta}), and  Propositions~\ref{prop Delta 2} and \ref{prop Delta 3}, the theorem holds when $\Delta(T) \leq 3$. Suppose $\Delta(T) \geq 4$. Since $\delta(T) \geq \left\lceil \frac{\Delta(T)}{2} \right\rceil$ (see (\ref{eq deltaDelta})), it suffices to show that $\delta(T) \leq \left\lceil \frac{\Delta(T)}{2} \right\rceil$.  
By Proposition~\ref{prop Delta 4}, there exists $a \in A(T)$ such that $\Delta({\rm Inv}(T,a)) = \Delta(T) - 2$. By the induction hypothesis, $\delta({\rm Inv}(T,a)) = \left\lceil \frac{\Delta({\rm Inv}(T,a))}{2} \right\rceil = \left\lceil \frac{\Delta(T)-2}{2} \right\rceil$. Since $\delta(T) \leq \delta({\rm Inv}(T,a)) +1$, we obtain $\delta(T) \leq \left\lceil \frac{\Delta(T)-2}{2} \right\rceil +1 = \left\lceil \frac{\Delta(T)}{2} \right\rceil$, as desired.  
\end{proof}

The rest of the paper aims to prove Propositions~\ref{prop Delta 2}, \ref{prop Delta 3} and \ref{prop Delta 4}. It is organized as follows. The next three sections contain the main preliminary results. Section~\ref{sections mod co-mod} contains the basic properties of modules and co-modules. Section~\ref{section mnc} contains a structural study of minimal co-modules. In Section~\ref{section delta decom}, we review some useful results about $\delta$-decompositions, i.e., $\Delta$-decompositions in which every element is a minimal co-module. Section~\ref{section proofs} is divided into two subsections. We prove Proposition~\ref{prop Delta 2} in Subsection~\ref{subsection proof1}, Propositions~\ref{prop Delta 3} and \ref{prop Delta 4} in Subsection~\ref{subsection proof2}. The way of algorithmic considerations is left open.        

\section{Modules and co-modules} \label{sections mod co-mod}

To manipulate modules of tournaments, it is convenient to introduce the following notations. Let $T$ be a tournament. For every distinct vertices $x,y \in V(T)$, we set   
\begin{equation*}
T(x,y)= \
\begin{cases}
1 \ \ \ \ \text{if} \ \  (x,y) \in A(T),\\

0 \ \ \ \   \text{if} \ \ (x,y) \notin A(T).
\end{cases}
\end{equation*} 
 Let $X$ and $Y$ be two disjoint subsets of $V(T)$. The notation $X \equiv_T Y$ signifies that $T(x,y) = T (x',y')$ for every $x, x' \in X$ and $y, y' \in Y$. 
For more precision when $X \equiv_T Y$, we write $T(X,Y) =1$ (resp. $T(X,Y) =0$) to indicate that for every $x \in X$ and $y \in Y$, we have $T(x,y) =1$ (resp. $T(x,y) =0$). When $X$ is a singleton $\{x\}$, we write $x \equiv_T Y$ for $\{x\} \equiv_T Y$, $T(x,Y)$ for $T(\{x\}, Y)$, and $T(Y,x)$ for $T(Y, \{x\})$.    
For instance, given a subset $M$ of $V(T)$, $M$ is a module of $T$ if and only if for every $x \in \overline{M}$, we have $x \equiv_T M$, or, equivalently, if and only if $T(x,u) = T(x,v)$ for every $x \in \overline{M}$ and $u,v \in M$.    

We now review some useful properties of the modules of a tournament. We begin by the following properties  which resemble those of the intervals in a total order. 

\begin{prop} \label{propo modules}
Let $T$ be a tournament.
\begin{enumerate}
\item Given a subset $W$ of $V(T)$, if $M$ is a module of $T$, then $M \cap W$ is a module of $T[W]$.
\item Given a module $M$ of $T$, if $N$ is a module of $T[M]$, then $N$ is also a module of $T$.
 \item If $M$ and $N$ are modules of $T$, then $M \cap N$ is also a module of $T$.
 \item If $M$ and $N$ are modules of $T$ such that $M \cap N \neq \varnothing$, then $M \cup N$ is also a module of $T$.
 \item If $M$ and $N$ are modules of $T$ such that $M \setminus N \neq \varnothing$, then $N \setminus M$ is also a module of $T$.
 \item If $M$ and $N$ are disjoint modules of $T$, then $M \equiv_T N$.
\end{enumerate}
\end{prop}   

Now we examine the modules of a tournament $T$ together with those of a tournament $T'$ obtained from $T$ by reversing an arc. We say that two sets $E$ and $F$ {\it overlap} if $E \cap F \neq \varnothing$, $E \setminus F \neq \varnothing$ and $F \setminus E \neq \varnothing$.

\begin{lem} \label{lem comodules}
Given a tournament $T$, consider an arc $a \in A(T)$ and let $T' = {\rm Inv}(T, a)$.   
\begin{enumerate}
 \item Given a module $M$ of $T$, $M$ is a module of $T'$ if and only if $M$ and $\mathcal{V}(a)$ do not overlap. 
\item If $M$ is a module of $T$ and $M'$ is a module of $T'$, then $M \cap M'$ is a module of $T$ or of $T'$. 
 \item Given a module $M$ of $T$ and a module $M'$ of $T'$ such that $M \cap M' \neq \varnothing$, if $M \cup M'$ and $\mathcal{V}(a)$ do not overlap, then $M \cup M'$ is a module of both $T$ and $T'$.
\end{enumerate}
\end{lem}

\begin{proof}
The first assertion is obvious because for every distinct $x, y \in V(T)$, $T'(x,y) = T(x,y)$ if and only if $\{x,y\} \neq \mathcal{V}(a)$. Let $M$ be a module of $T$, and let $M'$ be a module of $T'$. 

To prove the second assertion, we first suppose that $M$ and $\mathcal{V}(a)$ do not overlap. By the first assertion, $M$ is also a module of $T'$. Therefore, $M \cap M'$ is a module of $T'$ by Assertion~3 of Proposition~\ref{propo modules}. Now suppose that $M$ and $\mathcal{V}(a)$ overlap. In this instance, $T'[M] = T[M]$. Therefore, $M \cap M'$ is a module of $T[M]$ because $M \cap M'$ is a module of $T'[M]$ by Assertion~1 of Proposition~\ref{propo modules}. By Assertion~2 of Proposition~\ref{propo modules}, $M \cap M'$ is also a module of $T$.

For the third assertion, suppose that $M \cap M' \neq \varnothing$ and that $\mathcal{V}(a)$ do not overlap $M \cup M'$. Since $T' = {\rm Inv}(T, \mathcal{V}(a))$ and $T = {\rm Inv}(T', \mathcal{V}(a))$, we may interchange $T$ and $T'$ so that it suffices to show that $M \cup M'$ is a module of $T$. If $M'$ and $\mathcal{V}(a)$ do not overlap, then since $M'$ is also a module of $T$ by the first assertion of the lemma, $M \cup M'$ is a module of $T$ by Assertion~4 of Proposition~\ref{propo modules}. Suppose that $M'$ and $\mathcal{V}(a)$ overlap. Since $M \cup M'$ and $\mathcal{V}(a)$ do not overlap, then $\mathcal{V}(a) \subseteq M \cup M'$. Let $v \in \overline{M \cup M'}$. 

\begin{equation} \label{eq MM'}
\text{For every} \ u \in M \cup M', \ \text{we have} \ T(v,u) = T'(v,u) 
\end{equation}
because $\{u,v\} \neq \mathcal{V}(a)$.  Let $x, y \in M \cup M'$. Fix $z \in M \cap M'$. Since $M$ and $M'$ are modules of $T$ and $T'$ respectively, it follows from (\ref{eq MM'}) that $T(v,x) = T(v,z)$ and $T(v,y)= T(v,z)$. Thus $T(v,x) = T(v,y)$. Therefore, $M \cup M'$ is a module of $T$.          
\end{proof}

We now review some useful properties of co-modules and co-modular decompositions.

\begin{lem} \label{comod part}
 Given a decomposable tournament $T$, consider a co-modular decomposition $D$ of $T$. The following assertions are satisfied.
 \begin{enumerate}
 \item The tournament $T$ admits at most two singletons which are co-modules of $T$. In particular, $D$ contains at most two singletons. 
  \item If $D$ contains an element $M$ which is not a module of $T$, then the elements of $D \setminus \{M\}$ are nontrivial modules of $T$.
  \item If $D$ is a $\Delta$-decomposition of $T$, then $\cup D$ is never included in a co-module of $T$.
  \item If $D$ is a $\Delta$-decomposition of $T$ and $v(T) \geq 4$, then $D$ contains a nontrivial module of $T$.
 \end{enumerate}
\end{lem}

\begin{proof}
For the first assertion, suppose there are distinct $x,y \in V(T)$ such that $\{x\}$ and $\{y\}$ are co-modules of $T$, i.e., $\overline{\{x\}}$ and $\overline{\{y\}}$ are nontrivial modules of $T$. Let $z \in V(T) \setminus \{x,y\}$. We have $T(x,z) = T(x,y)$ and $T(y,x) = T(y,z)$. Thus $T(x,z) \neq T(y,z)$. Therefore, $\{z\}$ is a not a co-module of $T$. Since the elements of $D$ are co-modules of $T$, it follows that $D$ contains at most two singletons. 

For the second assertion, suppose that $D$ contains an element $M$ which is not a module of $T$. Let $N$ be an element of $D \setminus \{M\}$. We have to prove that $N$ is a nontrivial module of $T$. Suppose not. Since $\overline{M}$ and $\overline{N}$ are modules of $T$ and $\overline{M} \setminus \overline{N} =N \neq \varnothing$, then $\overline{N} \setminus \overline{M} =M$ is a module of $T$ by Assertion~5 of Proposition~\ref{propo modules}, a contradiction. Thus, $N$ is a nontrivial module of $T$. 

The third assertion holds because if $\cup D$ is included in a co-module $M$ of $T$, then $D \cup \{\overline{M}\}$ is a co-modular decomposition of $T$ so that $D$ is not a $\Delta$-decomposition of $T$. 

To prove the fourth assertion, suppose that $D$ is a $\Delta$-decomposition of $T$, and that $v(T) \geq 4$. Recall that $|D| = \Delta(T) \geq 2$ because $T$ is decomposable (see (\ref{eq dec delta})). Suppose for a contradiction that $D$ does not contain a nontrivial module of $T$. By the second assertion of the lemma, all the elements of $D$ are trivial modules of $T$. Since the elements of $D$ are co-modules of $T$, it follows that they are singletons (see (\ref{eq comod non vide})). By the first assertion of the lemma, $D = \{\{x\}, \{y\}\}$ for some distinct $x, y \in V(T)$. In particular $\Delta(T) =2$. Since $\overline{\{x\}}$ and $\overline{\{y\}}$ are modules of $T$, then by Assertion~3 of Proposition~\ref{propo modules}, $\overline{\{x\}} \cap \overline{\{y\}}=\overline{\{x,y\}}$ is a module of $T$. Moreover the module  $\overline{\{x,y\}}$ of $T$ is nontrivial because $v(T) \geq 4$. Therefore, $\{x,y\}$ is a co-module of $T$, which contradicts the third assertion of the lemma. Thus, $D$ contains a nontrivial module of $T$. 
\end{proof}

\section{Minimal co-modules} \label{section mnc}

Let $T$ be a tournament. A {\it minimal co-module} of $T$ is a co-module $M$ of $T$ which is minimal in the set of co-modules of $T$ ordered by inclusion, i.e., such that $M$ does not contain any other co-module of $T$.

\begin{Notation} \normalfont
  Given a tournament $T$, the set of minimal co-modules of $T$ is denoted by $\text{mc}(T)$.
\end{Notation}

For example, in the case of transitive tournaments, for every integer $n \geq 3$, we have (see (\ref{eq mod inter n}))
\begin{equation} \label{mnc(n)}
 \text{mc}(\underline{n}) = \{\{0\}, \{n-1\}\} \cup \{\{i,i+1\}: 1 \leq i \leq n-3\}.
\end{equation}

The following remark is the analogue of Assertion~1 of Lemma~\ref{lem comodules} for minimal co-modules.

\begin{rem} \label{rem mc overlap} \normalfont
 Let $T$ be a tournament, let $a \in A(T)$, and let $T' = {\rm Inv}(T, a)$. Given $M \subseteq V(T) \setminus \mathcal{V}(a)$, $M \in \text{mc}(T)$ if and only if $M \in  \text{mc}(T')$.
\end{rem}

Similarly, a {\it minimal nontrivial module} of a tournament $T$ is a nontrivial module of $T$ which is minimal in the set of nontrivial modules of $T$ ordered by inclusion.

\begin{rem} \label{fact1 min mod comod} \normalfont
Given a nontrivial module $M$ of a tournament $T$, if $M$ is a minimal co-module of $T$, then $M$ is a minimal nontrivial module of $T$. 
\end{rem}

The next remark is an immediate consequence of Assertion~2 of Proposition~\ref{propo modules}.

\begin{rem} \label{fact2 min mod indec} \normalfont
 If $M$ is a minimal nontrivial module of a tournament $T$, then $T[M]$ is indecomposable. 
\end{rem}

Given a tournament $T$, the elements of $\text{mc}(T)$ are clearly pairwise incomparable with respect to inclusion. Therefore, given distinct $M, N \in \text{mc}(T)$, either $M \cap N = \varnothing$ or $M$ and $N$ overlap. To study the overlapping case, we need the following notations. 

 \begin{Notation} \label{not oO} \normalfont
  Let $T$ be a tournament, and let $M \in \text{mc}(T)$. The set of the elements $N \in \text{mc}(T)$ that overlap $M$ is denoted by $O_T(M)$, i.e., $O_T(M) = \{N \in \text{mc}(T) : N \ \text{overlaps } \ M\}$. Moreover, we set $o_T(M) = |O_T(M)|$.  
 \end{Notation}
 
 For example, in the case of transitive tournaments, we obtain the following fact.
 
 \begin{Fact} \label{fact ot(n)}
  Consider the transitive tournament $\underline{n}$, where $n \geq 3$. We have ${\rm mc}(\underline{n}) = \{\{0\}, \{n-1\}\} \cup \{\{i,i+1\}: 1 \leq i \leq n-3\}$ (see (\ref{mnc(n)})). Moreover, we have the following.
  \begin{itemize}
   \item $o_{\underline{n}}(\{0\}) = o_{\underline{n}}(\{n-1\}) = 0$.
   \item Suppose $n \geq 5$. We have $o_{\underline{n}}(\{1,2\}) = o_{\underline{n}}(\{n-2,n-3\}) = 1$. More precisely, $O_{\underline{n}}(\{1,2\}) = \{2,3\}$ and $O_{\underline{n}}(\{n-2,n-3\}) = \{n-3,n-4\}$.
   \item Suppose $n \geq 6$, and let $i \in \{2, \ldots, n-4\}$. We have $o_{\underline{n}}(\{i,i+1\}) =2$. More precisely, $O_{\underline{n}}(\{i,i+1\}) = \{\{i-1,i\}, \{i+1,i+2\}\}$.  
  \end{itemize}
\end{Fact}

 The next result (see Lemma~\ref{degre G_T}) leads us to distinguish the modules with two vertices as specific modules.

 \begin{defn} \normalfont
  A {\it twin} of a tournament $T$ is a module of cardinality 2 of $T$.
 \end{defn}

 \begin{lem} \label{degre G_T}
Let $T$ be a tournament and let $M \in {\rm mc}(T)$. 
 We have $o_T(M) \leq 2$. Moreover, if $M$ is not a twin of $T$, then $o_T(M) = 0$.
  \end{lem}
  
 \begin{proof}
 To begin, suppose $o_{T}(M) \neq 0$. We will prove that $M$ is a twin of $T$. Let $N \in O_T(M)$.
 
 If $M \cup N = V(T)$, then $\overline{M} = N \setminus M$ is a co-module of $T$ (see (\ref{eq close})), which contradicts the minimality of the co-module $N$ of $T$ because $M \cap N \neq \varnothing$ (see Notation~\ref{not oO}). Thus 
 \begin{equation} \label{eq lem 3.1}
 M \cup N \neq V(T).
 \end{equation}
 Suppose for a contradiction that $M$ is not a module of $T$. 
 Since $\overline{M}$ is a module of $T$, $N$ or $\overline{N}$ is a module of $T$, $\overline{M} \cap \overline{N} \neq \varnothing$ (see (\ref{eq lem 3.1})), and $\overline{M} \cap N \neq \varnothing$ because $N \in O_T(M)$, then by Assertion~4 of Proposition~\ref{propo modules}, $\overline{M} \cup \overline{N}$ or $\overline{M} \cup N$ is a module of $T$. Moreover, since $M$ and $N$ overlap, then $2 \leq |\overline{M} \cup \overline{N}| \leq v(T)-1$ and $2 \leq |\overline{M} \cup N| \leq v(T)-1$. Thus, $\overline{M} \cup \overline{N}$ or $\overline{M} \cup N$ is a nontrivial module of $T$. In particular, $M \cap N$ or $M \cap \overline{N}$ is a co-module of $T$, which contradicts the minimality of the co-module $M$ of $T$.    
 Thus, $M$ is a module of $T$. Similarly, $N$ is also a module of $T$. By Assertions~3 and 5 of Proposition~\ref{propo modules}, $M \cap N$ and $M \setminus N$ are modules of $T$. Moreover, each of these modules is trivial because $M \in \text{mc}(T)$. Since $M$ and $N$ overlap, it follows that $|M \setminus N| = |M \cap N| =1$ and thus $|M| =2$. Since $M$ is a module of $T$, then it is a twin of $T$ as desired.
 
 Now we prove that $o_T(M) \leq 2$. Suppose not and consider three pairwise distinct elements $I,J,K$ of $O_{T}(M)$. As shown above, $M,I,J$ and $K$ are twins of $T$. By interchanging $I,J$ and $K$, we may assume $|M \cap I \cap J| = 1$. Thus, there are four pairwise distinct vertices $x,y,z,t \in V(T)$ such that $M = \{x,y\}$, $I = \{y,z\}$, and $J = \{y,t\}$. By Assertion~4 of Proposition~\ref{propo modules}, $\{x,y,z\}$, $\{x,y,t\}$ and $\{y,z,t\}$ are modules of $T$. Since $\{y,z,t\}$ is a module of $T$, by interchanging $T$ and $T^{\star}$, we may assume $T(x, \{y,z,t\}) =1$. Thus, $T(z, \{x,y,t\}) = 0$ because $T(z,x) =0$ and $\{x,y,t\}$ is a module of $T$. Hence $T(t,z) =1 \neq T(t,x) =0$, a contradiction because $\{x,y,z\}$ is a module of $T$. Thus $o_T(M) \leq 2$. 
\end{proof}

To continue the examination of minimal co-modules, 
we extend the notion of twin to that of transitive module (see Definition~\ref{def trans mod}).

\begin{defn} \label{def trans mod} \normalfont
Let $T$ be a tournament. A {\it transitive module} of $T$ is a module $M$ of $T$ such that the subtournament $T[M]$ is transitive. A transitive module $M$ of $T$ is {\it nontrivial} if the module $M$ of $T$ is nontrivial. A {\it transitive component} of $T$ is a transitive module of $T$ which is maximal (under inclusion) among the transitive modules of $T$.
 \end{defn}
 
For example, the modules with at most two vertices are transitive.

\begin{rem} \label{rem M union N transit} \normalfont
 Given a tournament $T$, consider two disjoint subsets $M$ and $N$ of $T$ such that the tournaments $T[M]$ and $T[N]$ are transitive. If $M \equiv_T N$, then the tournament $T[M \cup N]$ is transitive. In particular, if $M$ and $N$ are transitive modules of $T$, then the tournament $T[M \cup N]$ is transitive. 
\end{rem}

We need the next two results about transitive modules and transitive components. The following lemma is the analogue of Assertion~4 of Proposition~\ref{propo modules} for transitive modules.    

\begin{lem} \label{trans modules}
Given a tournament $T$, if $M$ and $N$ are transitive modules of $T$ such that $M \cap N \neq \varnothing$, then $M \cup N$ is also a transitive module of $T$.
\end{lem}

\begin{proof}
Let $M$ and $N$ be two transitive modules of $T$ such that $M \cap N \neq \varnothing$.
 By Assertion~4 of Proposition~\ref{propo modules}, $M \cup N$ is a module of $T$. We have to prove that the tournament $T[M \cup N]$ is transitive. 
 If $M$ and $N$ do not overlap, then we are done because in this instance, $M \cup N = M$ or $N$. Hence suppose that $M$ and $N$ overlap. By Assertion~5 of Proposition~\ref{propo modules}, $N \setminus M$ is a module of $T$. By Assertion~6 of Proposition~\ref{propo modules}, we have $M \equiv_T N \setminus M$. Moreover, the module $N \setminus M$ of $T$ is transitive because the module $N$ is. It follows that the tournament $T[M \cup N] = T[M \cup (N \setminus M)]$ is transitive (see Remark~\ref{rem M union N transit}).    
\end{proof}

\begin{cor} \label{trans partition}
 Given a tournament $T$, the transitive components of $T$ form a partition of $V(T)$.
\end{cor}

\begin{proof}
 Let $\mathcal{C}(T)$ be the set of the transitive components of the tournament $T$. Let $v \in V(T)$. The singleton $\{v\}$ is obviously a transitive module of $T$. Let $C_v$ be the union of all the transitive modules of $T$ containing $v$. By Lemma~\ref{trans modules}, $C_v$ is a transitive module of $T$. Thus, clearly $C_v \in \mathcal{C}(T)$. It follows that $V(T) \subseteq \cup \mathcal{C}(T)$. Now let $C$ and $C'$ be two elements of $\mathcal{C}(T)$. Suppose $C \cap C' \neq \varnothing$. Again by Lemma~\ref{trans modules}, $C \cup C'$ is a transitive module of $T$. It follows from the maximality of the transitive modules $C$ and $C'$ of $T$ that $C=C'$. Thus, $\mathcal{C}(T)$ is a partition of $V(T)$.        
\end{proof}

The following observation is a consequence of Lemmas~\ref{degre G_T} and \ref{trans modules}.

\begin{Observation} \label{o=2 tr=4}
 Let $T$ be a decomposable tournament and let $M \in {\rm mc}(T)$.
 If $o_T(M) =2$, then $M$ is contained in a transitive component $C$ of $T$ such that $|C| \geq 4$. 
\end{Observation}

\begin{proof}
 Suppose $o_T(M) =2$. Let $N$ and $L$ be the two distinct elements of $O_T(M)$. By Lemma~\ref{degre G_T}, $M$, $N$, and $L$ are twins of $T$. Let $C$ be the transitive component of $T$ such that $M \subseteq C$. By Lemma~\ref{trans modules}, $C$ is the union of the transitive modules of $T$ containing $M$. Moreover, again by Lemma~\ref{trans modules}, $M \cup N \cup L$ is a transitive module of $T$.
 It follows that $M \cup N \cup L \subseteq C$. To complete the proof, it suffices to verify that $|M \cup N \cup L| =4$.  
 Since $M$, $N$, and $L$ are pairwise distinct twins of $T[M \cup N \cup L]$ (see Assertion~1 of Proposition~\ref{propo modules}), and since the tournaments with three vertices do not admit three pairwise distinct twins, then $|M \cup N \cup L| \geq 4$, and thus $|M \cup N \cup L| =4$ because $O_T(M) = \{N,L\}$.            
\end{proof}

Now we will see how the minimal co-modules of a tournament $T$ are delimited by the transitive components of $T$, in the sense that an element of $\text{mc}(T)$ never overlaps a  transitive component of $T$ (see Lemma~\ref{lem C mnc}). To introduce a notation indicating the minimal co-modules that are contained in a transitive component (see Notation~\ref{not M(k)}), we have to use the following fact.    

\begin{Fact} \label{rem twin mnc} 
 Let $T$ be a tournament with at least three vertices. If $T$ admits a twin $W = \{x,y\}$, then $|{\rm mc}(T) \cap \{W, \{x\}, \{y\}\}| = 1$. 
\end{Fact}

\begin{proof}
Suppose that $T$ admits a twin $W = \{x,y\}$.
 If $W \in \text{mc}(T)$, then by minimality of $W$, $\text{mc}(T) \cap \{W, \{x\}, \{y\}\} = \{W\}$. Hence suppose $W \notin \text{mc}(T)$. In this instance, since $W$ is a co-module of $T$, then $\{x\} \in \text{mc}(T)$ or $\{y\} \in \text{mc}(T)$. By interchanging $x$ and $y$, we may assume $\{x\} \in \text{mc}(T)$. We have to prove that $\{y\} \notin \text{mc}(T)$. Let $z \in V(T) \setminus \{x,y\}$. We have $T(x,y) = T(x,z)$ because $\overline{\{x\}}$ is a module of $T$. Moreover, $T(x,z) = T(y,z)$ because $\{x,y\}$ is a module of $T$. Thus $T(x,y) = T(y,z)$, i.e., $T(y,x) \neq T(y,z)$. Therefore, $\overline{\{y\}}$ is not a module of $T$ and thus $\{y\} \notin \text{mc}(T)$.       
\end{proof}

\begin{Notation} \label{not M(k)}\normalfont
Let $T$ be a tournament with at least three vertices. Suppose that $T$ admits a transitive component $C$ such that $|C|=n \geq 2$. Let us denote the elements of $C$ by $v_0, \ldots, v_{n-1},$ in such a way that $T[C] = (C, \{(v_i,v_j): 0 \leq i < j \leq n-1\})$. For every $k \in \{0, \ldots, n-2\}$, the pair $\{v_{k}, v_{k+1}\}$ is a twin of $T[C]$ (see (\ref{eq mod inter n})) and thus of $T$ (see Assertion~2 of Proposition~\ref{propo modules}). The unique element of ${\rm mc}(T)$ that is contained in $\{v_{k}, v_{k+1}\}$ (see Fact~\ref{rem twin mnc}) is denoted by $C(k)$.   
\end{Notation}

\begin{Example} \label{example trans1} \normalfont
Consider the case where $T$ is the transitive tournament $\underline{n}$, where $n \geq 3$. The unique transitive component of $T$ is $C = V(T) = \{0, \ldots, n-1\}$. We have $C(0) = \{0\}$, $C(n-2)= \{n-1\}$, and for every integer $k$ such that $1 \leq k \leq n-3$, we have $C(k) = \{k, k+1\}$.
\end{Example}

Observation~\ref{obs M(k)} contains more details about $C(k)$ in the general case.

\begin{Observation} \label{obs M(k)}
 Let $T$ be a tournament with at least four vertices. Suppose that $T$ admits a transitive component $C$ such that $|C|=n \geq 3$. Let us denote the elements of $C$ by $v_0, \ldots, v_{n-1},$ in such a way that $T[C] = (C, \{(v_i,v_j): 0 \leq i < j \leq n-1\})$. The following assertions are satisfied.
 \begin{enumerate}
  \item For every $k \in \{0, \ldots, n-2\}$, we have
\begin{equation*}
C(k) = \
\begin{cases}
\{v_0\} \ \text{or} \ \{v_0, v_1\} \ \ \ \ \ \ \ \ \ \ \ \ \  \text{if} \ \  k=0,\\

\{v_{n-1}\} \ \text{or} \ \{v_{n-2},v_{n-1}\} \  \ \ \ \  \text{if} \ k= n-2,\\        

\{v_k, v_{k+1}\} \ \ \ \ \ \ \ \ \ \ \ \ \ \ \ \ \ \ \ \ \hspace{0.09cm} \text{otherwise.}
\end{cases}
\end{equation*}
\item We have $C = \displaystyle\bigcup_{k=0}^{n-2} C(k)$. 
 \end{enumerate}
\end{Observation}

\begin{proof}
 To verify the first assertion, let $k \in \{0, \ldots, n-2\}$. Since $T(v_1,v_2) = 1 \neq T(v_1,v_0) =0$ (resp. $T(v_{n-2},v_{n-1}) = 1 \neq T(v_{n-2},v_{n-3}) =0$), then $\overline{\{v_1\}}$ (resp. $\overline{\{v_{n-2}\}}$) is not a module of $T$. Therefore, neither $\{v_1\}$ nor $\{v_{n-2}\}$ is a co-module of $T$. It follows from the definitions of $C(0)$ and $C(n-2)$ that $C(0) \in \{\{v_0\}, \{v_{0}, v_{1}\}\}$ and $C(n-2) \in \{\{v_{n-1}\}, \{v_{n-2}, v_{n-1}\}\}$. Now suppose $1 \leq k \leq n-3$. Since $T(v_k, v_{k+1}) =1 \neq T(v_{k}, v_{k-1}) =0$ (resp. $T(v_{k+1}, v_{k+2}) =1 \neq T(v_{k+1}, v_{k}) =0$), then $\overline{\{v_k\}}$ (resp. $\overline{\{v_{k+1}\}}$) is not a module of $T$. Therefore, neither $\{v_{k}\}$ nor $\{v_{k+1}\}$ is a co-module of $T$. It follows from the definition of $C(k)$ that $C(k) = \{v_{k}, v_{k+1}\}$. 
 
 We now verify the second assertion. If $n \geq 4$, the second assertion is an immediate consequence of the first one. Hence suppose $n=3$. We have $C = \{v_0, v_1, v_2\}$. By the first assertion $C(0) \in \{\{v_0\}, \{v_0, v_{1}\}\}$ and $C(1) \in \{\{v_2\}, \{v_1, v_{2}\}\}$. 
 Suppose for a contradiction that $C(0) = \{v_0\}$ and $C(1) = \{v_2\}$. In this instance, $\overline{\{v_0\}}$ and $\overline{\{v_2\}}$ are modules of $T$. 
 Thus $T(v_0, \overline{C}) = 1 \neq T(v_2, \overline{C}) = 0$, contradicting that $C$ is a nontrivial module of $T$. It follows that $C(0) = \{v_0, v_1\}$ or $C(1) = \{v_1, v_2\}$. Thus $C = C(0) \cup C(1)$, as desired.   
\end{proof}

\begin{lem} \label{lem C mnc}
 Let $T$ be a tournament with at least three vertices. Suppose that $T$ admits a transitive component $C$ such that $|C|=n \geq 2$. Given $M \subseteq V(T)$, the following assertions are equivalent.
 \begin{enumerate}
  \item $M \in {\rm mc}(T)$ and $M \cap C \neq \varnothing$.
  \item $M \in \{C(0), \ldots, C(n-2)\}$.
 \end{enumerate}
\end{lem}

\begin{proof}
If the tournament $T$ is transitive, then $C = V(T)$ and the lemma follows immediately from (\ref{mnc(n)}) and Example~\ref{example trans1}. Hence suppose that $T$ is non-transitive. In this instance, $C$ is a nontrivial transitive component of $T$. In particular, $|C| \neq v(T)-1$ (see Remark~\ref{rem M union N transit}). Thus 
\begin{equation} \label{eq correction}
2 \leq |C|=n \leq v(T)-2. 
\end{equation}
The second assertion implies the first one by the definition of $C(k)$ for $k \in \{0, \ldots, n-2\}$ (see Notation~\ref{not M(k)}). Conversely, suppose $M \in {\rm mc}(T)$ and $M \cap C \neq \varnothing$.
 Up to isomorphism, we may assume $T[C] = \underline{n}$.
 
 First suppose $n \geq 3$. In this instance, we have $C = \displaystyle\bigcup_{k=0}^{n-2} C(k)$ (see Assertion~2 of Observation~\ref{obs M(k)}). Thus, there is $k \in \{0, \ldots, n-2\}$ such that $M \cap C(k) \neq \varnothing$. If $M = C(k)$, then we are done. Hence suppose $M \neq C(k)$. In this instance, since $M$ and $C(k)$ are distinct elements of $\text{mc}(T)$ and $M \cap C(k) \neq \varnothing$, then $M$ and $C(k)$ overlap. It follows from Lemma~\ref{degre G_T} that $M$ is a twin of $T$. By Lemma~\ref{trans modules}, $C \cup M$ is a transitive module of $T$. By maximality of the transitive module $C$ of $T$, we obtain $M \subseteq C$. 
 Since $M$ is a twin of $T$ and $M \subseteq C$, then by Assertion~1 of Proposition~\ref{propo modules}, $M$ is a also a twin of $T[C] = \underline{n}$. Thus, $M = \{i,i+1\}$ for some $i \in \{0, \ldots, n-2\}$ (see (\ref{eq mod inter n})). Since $M \in \text{mc}(T)$, it follows that $M = C(i)$ (see Notation~\ref{not M(k)}).
 
 Second suppose $n = 2$. We have to prove that $M = C(0)$. Since $T[C] = \underline{2}$, we have $C(0) = \{0\}, \{1\}$, or $\{0,1\}$. By interchanging the vertices $0$ and $1$, as well as the tournaments $T$ and $T^{\star}$, we may assume that $C(0) = \{0\}$ or $C(0) = C =\{0,1\}$. 
 
 To begin, suppose $C(0) = C =\{0,1\}$. In this instance, $C \in {\rm mc}(T)$.    
 For a contradiction, suppose that $M$ overlaps $C$. By Lemma~\ref{degre G_T}, $M$ is a twin of $T$. It follows from Lemma~\ref{trans modules} that $C \cup M$ is a transitive module of $T$. This contradicts the maximality of the transitive module $C$ of $T$. Thus, $M$ and $C$ do not overlap. Therefore, since $M$ and $C$ are minimal co-modules of $T$ and $M \cap C \neq \varnothing$, we obtain $M =C = C(0)$.
 
 Finally, suppose $C(0) = \{0\}$. By minimality of the co-modules $M$ and $C(0)$ of $T$, we have $M = C(0) = \{0\}$ or $M \cap C = \{1\}$. Suppose for a contradiction that $M \cap C = \{1\}$. Recall that since $\overline{\{0\}}$ and $C = \{0,1\}$ are modules of $T$, and $T[C] = \underline{2}$, then $T(0, \overline{\{0\}}) =1$ and
 \begin{equation} \label{eq ccbar}
 T(C, \overline{C}) = 1. 
 \end{equation}
 Therefore, $\{1\}$ is not a  co-module of $T$ because $T(1,0) = 0 \neq T(1, \overline{C}) =1$. It follows that $M \neq \{1\}$ and since $M \cap C = \{1\}$, we have $M \setminus C \neq \varnothing$. Consider a vertex $x \in M \setminus C$. Since the tournament $T[C \cup \{x\}]$ is transitive (see Remark~\ref{rem M union N transit}), it follows from the maximality of the transitive module $C$ of $T$ that $C \cup \{x\}$ is not a module of $T$. 
 Moreover, since $T(C, \overline{C}) = 1$ (see (\ref{eq ccbar})) and $C \cup \{x\}$ is not a module of $T$, there exists a vertex $y \in V(T) \setminus (C \cup \{x\})$ such that $T(y,x) = 1$. It follows that if $M$ is a module of $T$, then $y \in M$, and by Assertion~5 of Proposition~\ref{propo modules}, $M \setminus C$ is a nontrivial module of $T$, which contradicts $M \in {\rm mc}(T)$. Thus, $M$ is not a module of $T$. Since $\overline{M}$ is a module of $T$, we obtain $T(1, \overline{M}) = 0$ because $T(1,0) =0$, $0 \in \overline{M}$, and $1 \notin \overline{M}$. But $T(1, \overline{C}) =1$ (see (\ref{eq ccbar})). It follows that $\overline{M} = \{0\}$. In particular $\overline{C} \varsubsetneq M$, which contradicts $M \in {\rm mc}(T)$ because $\overline{C}$ is a nontrivial module of $T$ (see (\ref{eq ccbar}) and (\ref{eq correction})). We conclude that $M = C(0) = \{0\}$, completing the proof.   
\end{proof}
 
The following fact is an immediate consequence of Lemma~\ref{lem C mnc}.

\begin{Fact} \label{o(0,n-2)}
 Given a tournament $T$ with at least three vertices, if $T$ admits a transitive component $C$ such that $|C| \geq 2$, then $o_T(C(0)) \leq 1$ and $o_T(C(|C|-2)) \leq 1$. 
\end{Fact}

\section{Minimal $\Delta$-decompositions (or $\delta$-decompositions)} \label{section delta decom}

Minimal $\Delta$-decompositions form a basic tool in our proofs of Propositions~\ref{prop Delta 2}, \ref{prop Delta 3} and \ref{prop Delta 4}. A {\it minimal $\Delta$-decomposition} (or a  {\it $\delta$-decomposition}) of a tournament $T$ is a $\Delta$-decomposition $D$ of $T$ in which every element is a minimal co-module of $T$, i.e., such that $D \subseteq \text{mc}(T)$. To see that every tournament $T$ admits a $\delta$-decomposition, let $D$ be a $\Delta$-decomposition of $T$. For every element $M$ of $D$, since $M$ is a co-module of $T$, there exists a minimal co-module $M^{-}$ of $T$ such that $M^{-} \subseteq M$. Clearly $\{M^{-} : M \in D\}$ is a $\delta$-decomposition of $T$. For example, consider the case of transitive tournaments. 
The $\delta$-decompositions of $\underline{n}$ are the sets of maximum size among the subsets of $\text{mc}(\underline{n})$ whose elements are pairwise disjoint. Therefore, we obtain the following fact by using (\ref{mnc(n)}).

\begin{Fact} \label{fact deltadecom n}
 Consider the transitive tournament $\underline{n}$, where $n \geq 4$. The following two assertions hold.
 \begin{enumerate}
  \item If $n$ is even, then $\{\{0\}, \{n-1\}\} \cup \{\{2i-1,2i\}: 1 \leq i \leq \frac{n-2}{2}\}$ is the unique $\delta$-decomposition of $\underline{n}$. 
  \item If $n$ is odd, then a subset $D$ of $2^{\{0, \ldots, n-1\}}$ is a $\delta$-decomposition of $\underline{n}$ if and only if $D$ is a $\delta$-decomposition of $\underline{n} - i$ for some odd integer $i \in \{1, \ldots, n-2\}$.  
 \end{enumerate}
In particular, $\Delta(\underline{n}) = \left\lceil \frac{n+1}{2} \right\rceil$ as found in Proposition~\ref{Deltatr}.
\end{Fact}

The next remark is an immediate consequence of Assertion~4 of Lemma~\ref{comod part}.

\begin{rem} \label{rem cherifa} \normalfont
 Given a decomposable tournament $T$ such that $v(T) \geq 4$, every $\delta$-decomposition of $T$ contains a nontrivial module of $T$. In particular, $T$ admits a minimal co-module which is a nontrivial module of $T$. 
\end{rem}

The starting points of our proofs of Propositions~\ref{prop Delta 2}, \ref{prop Delta 3} and \ref{prop Delta 4} are based on the following result. 

\begin{prop} \label{propo Delta 234}
 Given a (decomposable) tournament $T$, the following three assertions are satisfied.
 \begin{enumerate}
  \item If $\Delta(T) = 2$, then for every $M \in {\rm mc}(T)$, we have $o_T(M) \leq 1$.
  \item If $\Delta(T) = 3$, then $T$ admits a $\delta$-decomposition $D$ such that $o_T(M) \leq 1$ for every $M \in D$.
  \item If $\Delta(T) \geq 4$, then $T$ admits a $\delta$-decomposition which contains four elements $M_1, M_2, M_3$ and $M_4$ satisfying the following conditions.
  \begin{enumerate}
  \item For every $i \in \{1,3,4\}$, $o_T(M_i) \leq 1$.
  \item $T(M_1,M_2) = T(M_2,M_3) = 1$.
  \item There exists $x \in M_4$ such that $T(x,M_1) =1$ or $T(M_3,x) =1$.
 \end{enumerate}
 \end{enumerate}
\end{prop}

The aim of the rest of this section is to prove Proposition~\ref{propo Delta 234}. For this purpose, we need some preliminary results. The following three ones are about $\delta$-decompositions. They are principally consequences of Lemmas~\ref{degre G_T} and \ref{lem C mnc}.   

\begin{cor} \label{if M notin D}
Given a decomposable tournament $T$, consider a $\delta$-decomposition $D$ of $T$ and let $M \in {\rm mc}(T)$. If $M \notin D$, then $o_T(M) \in \{1,2\}$ and $D \cap O_T(M) \neq \varnothing$.
\end{cor}

\begin{proof}
 Suppose that $o_T(M) \notin \{1,2\}$ or $D \cap O_T(M) = \varnothing$. By Lemma~\ref{degre G_T}, we have $o_T(M) = 0$ or $D \cap O_T(M) = \varnothing$. In both instances, if $M \notin D$, then $D \cup \{M\}$ would be a co-modular decomposition of $T$, which contradicts the hypothesis that $D$ is a $\delta$-decomposition of $T$. Thus $M \in D$. 
\end{proof}
 
 \begin{cor} \label{0, n-2 dans D}
 Let $T$ be a tournament. There exists a $\delta$-decomposition $D$ of $T$ such that for every transitive component $C$ of $T$ with $|C| \geq 4$, we have $\{C(0), C(|C|-2)\} \subseteq D$.
\end{cor}

\begin{proof}
 Consider a $\delta$-decomposition $D$ of $T$. Let $C$ be a transitive component of $T$ such that $|C| \geq 4$. By Assertion~1 of Observation~\ref{obs M(k)} and by Lemma~\ref{lem C mnc}, we have $O_T(C(0)) = \varnothing$ or $O_T(C(0)) = \{C(1)\}$. In the first instance, $C(0) \in D$ by Corollary~\ref{if M notin D}. In the second one, again by Corollary~\ref{if M notin D}, if $C(0) \notin D$, then $C(1) \in D$. 
 Similarly, we have $O_T(C(|C|-2)) = \varnothing$ or $O_T(C(|C|-2)) = \{C(|C|-3)\}$. In the first instance, $C(|C|-2) \in D$. In the second one, if $C(|C|-2) \notin D$, then $C(|C|-3) \in D$. To summarize, we have shown that for every transitive component $C$ of $T$ such that $|C| \geq 4$, the following two claims hold.
 \begin{enumerate}
  \item If $C(0) \notin D$, then $C(1) \in D$ and $O_T(C(0)) = \{C(1)\}$.
  \item If $C(|C|-2) \notin D$, then $C(|C|-3) \in D$ and $O_T(C(|C|-2)) = \{C(|C|-3)\}$.
 \end{enumerate}

 Now let $\mathcal{C}$ (resp. $\mathcal{C}'$) be the set of the transitive components $C$ of $T$ such that $|C| \geq 4$ and $C(0) \notin D$ (resp. $|C| \geq 4$ and $C(|C|-2) \notin D$). It follows from Claim~1 (resp. Claim~2) above that for every $C \in \mathcal{C}$ (resp. $C \in \mathcal{C}'$), we have $C(1) \in D$ and $O_T(C(0)) = \{C(1)\}$ (resp. $C(|C|-3) \in D$ and $O_T(C(|C|-2)) = \{C(|C|-3)\}$). Therefore, by taking $D'=(D \cup \{C(0): C \in \mathcal{C}\} \cup \{C(|C|-2): C \in \mathcal{C}'\}) \setminus (\{C(1) : C \in \mathcal{C}\} \cup \{C(|C|-3) : C \in \mathcal{C}'\})$, we obtain that $D'$ is a $\delta$-decomposition of $T$ such that for every transitive component $C$ of $T$ with $|C| \geq 4$, we have $\{C(0), C(|C|-2)\} \subseteq D'$, as desired.    
\end{proof}

\begin{cor} \label{exterieur component}
 Given a tournament $T$ admitting a nontrivial transitive component $C$, there exists  $M \in {\rm mc}(T)$ such that $M \cap C = \varnothing$ and $o_T(M) \leq 1$. 
\end{cor}

\begin{proof}
 Since $\overline{C}$ is a co-module of $T$, there exists $N \in \text{mc}(T)$ such that $N \subseteq \overline{C}$. By Lemma~\ref{degre G_T}, we can suppose $o_T(N) =2$. By Observation~\ref{o=2 tr=4}, $N$ is contained in a transitive component $C'$ of $T$ such that $|C'| \geq 4$. By Corollary~\ref{trans partition}, $C \cap C' = \varnothing$ and thus $C \cap C'(0) = \varnothing$. By Fact~\ref{o(0,n-2)}, $o_T(C'(0)) \leq 1$. Thus, it suffices to take $M = C'(0)$.    
 \end{proof}

We also need the following observation for the proof of Assertion~3 of Proposition~\ref{propo Delta 234}.

\begin{Observation} \label{fact config4}
 Given a tournament $T$ such that $\Delta(T) \geq 4$, consider a co-modular decomposition $D$ of $T$ such that $|D|=4$. The elements of $D$ can be denoted by $M_1, M_2, M_3, M_4,$ in such a way that $T(M_1,M_2) = T(M_2,M_3) =1$, and there exists $x \in M_4$ such that $T(x,M_1) =1$ or $T(M_3,x) =1$.   
\end{Observation}
\begin{proof}
 By Assertion~2 of Lemma~\ref{comod part}, the elements of $D$ can be denoted by $M_1, M_2, M_3, M_4,$ in such a way that $M_1, M_2$ and $M_3$ are modules of $T$. By interchanging $M_1, M_2$ and $M_3$, we may assume $T(M_1,M_2) = T(M_2,M_3) =1$ (see Assertion~6 of Proposition~\ref{propo modules}). If $T(x,M_1) =1$ or $T(M_3,x) =1$ for some $x \in M_4$, then we are done. Hence suppose $T(M_1,M_4) = T(M_4,M_3) =1$. In this instance $\overline{M_4}$ is not a module of $T$ because $T(M_4,M_1) =0 \neq T(M_4,M_3)=1$. Thus, $M_4$ is a module of $T$. Since $M_2$ and $M_4$ are modules of $T$, then by interchanging them, we may assume $T(M_2,M_4) =1$ (see Assertion~6 of Proposition~\ref{propo modules}). Thus $T(M_1,M_2) = T(M_2,M_4) = T(M_4,M_3) =1$. This completes the proof.     
\end{proof}

\begin{proof}[Proof of Proposition~\ref{propo Delta 234}]
It is straightforward to verify the proposition for the tournament with at most four vertices. In the rest of the proof, we suppose $v(T) \geq 5$. 

First suppose that $o_T(M) \leq 1$ for every $M \in \text{mc}(T)$. In this instance, the first two assertions are obviously satisfied. The third one follows from Observation~\ref{fact config4}.

Second suppose that the tournament $T$ is transitive. Up to isomorphism, we may assume $T = \underline{n}$ for some integer $n \geq 5$. Suppose $n=5$. By Fact~\ref{fact deltadecom n}, $\Delta(T) =3$ and $\{\{0\}, \{1,2\}, \{4\}\}$ is a $\delta$-decomposition of $T$. Moreover, $o_{T}(\{0\}) = o_T(\{4\}) =0$ and $o_T(\{1,2\}) =1$ (see Fact~\ref{fact ot(n)}). Thus, the second assertion is satisfied. Hence suppose $n \geq 6$. By Fact~\ref{fact deltadecom n}, $\Delta(T) \geq 4$ and $T$ admits a $\delta$-decomposition which contains $M_1 = \{0\}$, $M_2 = \{1,2\}$, $M_3 = \{n-2,n-3\}$, and $M_4 = \{n-1\}$. We have $T(M_1,M_2) = T(M_2,M_3) = T(M_3,M_4) =1$. Moreover, $o_T(M_1) = o_T(M_4) =0$ and $o_T(M_2) = o_T(M_3)=1$ (see Fact~\ref{fact ot(n)}). Thus, the third assertion is satisfied. We conclude that the proposition holds for transitive tournaments.  

Third, suppose that the tournament $T$ is non-transitive, and that there exists $Y \in \text{mc}(T)$ such that $o_T(Y) \geq 2$. By Lemma~\ref{degre G_T}, $o_T(Y) =2$. By Observation~\ref{o=2 tr=4}, $Y$ is contained in a transitive component $C$ of $T$ such that $|C| =n \geq 4$. Moreover, $C \neq V(T)$ because $T$ is non-transitive. Therefore, $\{C(0), C(n-2), \overline{C}\}$ is a co-modular decomposition of $T$. In particular $\Delta(T) \geq 3$. By Fact~\ref{o(0,n-2)}, we have 
 \begin{equation} \label{11111}
 o_T(C(0)) \leq 1 \ \text{and} \  o_T(C(n-2)) \leq 1. 
 \end{equation}
 Moreover,
 since $C$ is a nontrivial transitive component of $T$, then by Corollary~\ref{exterieur component} 
 there exists $N \in \text{mc}(T)$ such that $N \cap C = \varnothing$ and $o_T(N) \leq 1$.
 It follows that if $\Delta(T) =3$, then $\{C(0), C(n-2), N\}$ is a $\delta$-decomposition of $T$ in which every element $X$ satisfies $o_T(X) \leq 1$, as desired.  
 Hence suppose $\Delta(T) \geq 4$.
 First, suppose that 
 \begin{equation} \label{22222}
 \text{there exists} \ Z \in {\rm mc}(T) \ \text{such that} \ Z \cap C = \varnothing \ \text{and} \ o_T(Z) =2. 
 \end{equation}
 By Observation~\ref{o=2 tr=4}, $Z$ is contained in a transitive component $C'$ of $T$ such that $|C'| =n' \geq 4$. By Corollary~\ref{trans partition}, $C \cap C' = \varnothing$. By Corollary~\ref{0, n-2 dans D}, there exists a $\delta$-decomposition $D$ of $T$ such that $\{C(0), C(n-2), C'(0), C'(n'-2)\} \subseteq D$. By interchanging the transitive components $C$ and $C'$, we may assume $T(C,C') =1$ (see Assertion~6 of Proposition~\ref{propo modules}). Thus, the third assertion is satisfied by taking the $\delta$-decomposition $D$ with its four elements $M_1 = C(0)$, $M_2 = C(n-2)$, $M_3 = C'(0)$, and $M_4 = C'(n'-2)$. Indeed,  $o_T(M_i) \leq 1$ for every $i \in \{1,2,3,4\}$ by Fact~\ref{o(0,n-2)}, and $T(M_1,M_2) = T(M_2,M_3 ) = T(M_3,M_4) =1$ by construction.
 
 Second, suppose that (\ref{22222}) does not hold. By Lemma~\ref{degre G_T},  
 \begin{equation} \label{33333}
 \text{for every} \ Z \in {\rm mc}(T) \ \text{such that} \ Z \cap C = \varnothing, \ \text{we have} \  o_T(Z) \leq 1.
 \end{equation}
 By Corollary~\ref{0, n-2 dans D}, there exists a $\delta$-decomposition $D$ of $T$ such that $\{C(0), C(n-2)\} \subseteq D$. To begin, suppose that $D \cap \{C(i): 0 \leq i \leq n-2\} = \{C(0), C(n-2)\}$. By Lemma~\ref{lem C mnc}, since $\Delta(T) \geq 4$, there exist distinct $M,L \in D$ such that $(M \cup L) \cap C = \varnothing$. By (\ref{33333}), we have $o_T(M) \leq 1$ and $o_T(L) \leq 1$. Thus, $\{C(0), C(n-2), M, L\}$ is a co-modular decomposition of $T$ that is contained in the $\delta$-decomposition $D$ of $T$, and in which every element $X$ satisfies $o_T(X) \leq 1$ (see (\ref{11111})). Therefore, the third assertion is satisfied by applying Observation~\ref{fact config4} to $\{C(0), C(n-2), M, L\}$. To finish, suppose that there is an integer $i$ such that $1 \leq i \leq n-3$ and $C(i) \in D$. We have $\{C(0), C(i), C(n-2)\} \subseteq D$. Moreover, by maximality of $D$, it follows from Lemma~\ref{lem C mnc} and Corollary~\ref{exterieur component} that $D$ contains an element $K$ such that $K \cap C = \varnothing$. We have $o_T(K) \leq 1$ (see (\ref{33333})). Thus, the third assertion is satisfied by taking the $\delta$-decomposition $D$ with its four elements $M_1 = C(0)$, $M_2 = C(i)$, $M_3 = C(n-2)$, and $M_4 = K$. Indeed, $o_T(M_i) \leq 1$ for every $i \in \{1,3,4\}$ (see (\ref{11111})), $T(M_1,M_2) = T(M_2,M_3 ) =1$, and since $C$ is a module of $T$, then for $x \in M_4$, we have $T(x,C) =1$ or $T(C,x) =1$.       
\end{proof}

\section{Proofs of Propositions \ref{prop Delta 2}, \ref{prop Delta 3} and \ref{prop Delta 4}} \label{section proofs}

The proofs use the following notation. 

\begin{Notation} \label{Notation M} \normalfont
 Given a decomposable tournament $T$, consider a minimal co-module of $T$ such that $o_T(M) \leq 1$. When $o_T(M) = 1$, we denote by $M'$ the element of $O_T(M)$. We set 
   \begin{equation*}
\widetilde{M} = \
\begin{cases}
\ \ \  M  \hspace*{0.95cm} \text{if} \ \  o_T(M) = 0,\\

M \cap M' \ \ \ \   \text{if}  \ \ o_T(M) = 1. 
\end{cases}
\end{equation*}
Notice that $\widetilde{M} \neq \varnothing$. More precisely, if $o_T(M) = 1$, then $M$ and $M'$ are twins of $T$ and $|\widetilde{M}| =1$ (see Lemma~\ref{degre G_T}).
 \end{Notation}

For a better understanding of Notation~\ref{Notation M}, notice the following remark which is a consequence of Lemma~\ref{degre G_T}. 

\begin{rem} \label{rem1} \normalfont
Given a decomposable tournament $T$ and a minimal co-module $M$ of $T$ such that $o_T(M) \leq 1$,
one of the following holds
\begin{enumerate}
 \item $\widetilde{M} = M$,
 \item $T$ admits a module $H$ such that $T[H] \simeq \underline{3}$, $M$ is a twin of $T[H]$ and thus of $T$, and $\widetilde{M} = \{f(1)\}$ where $f$ is the isomorphism from $\underline{3}$ onto $T[H]$. 
\end{enumerate}
\end{rem}

\subsection{Proof of Proposition~\ref{prop Delta 2}} \label{subsection proof1}
We need the next three lemmas.

\begin{lem} \label{lem Delta=2}
 Given a tournament $T$ such that $\Delta(T) = 2$, consider a $\delta$-decomposition $\{M,N\}$ of $T$. The following assertions are satisfied.
 \begin{enumerate}
  \item We have $o_T(M) \leq 1$ and $o_T(N) \leq 1$ so that $\widetilde{M}$ and $\widetilde{N}$ are well-defined.
  \item Let $x \in \widetilde{M}$ and $y \in \widetilde{N}$. Suppose that the tournament $T'={\rm Inv}(T,\{x,y\})$ is decomposable. Let $L$ be a minimal co-module of $T'$ which is a nontrivial module of $T'$. The following two assertions hold. 
  \begin{enumerate}
  \item[$2.1$.] $L$ and $\{x,y\}$ overlap.
  \item[$2.2$.] We have $L \cap M \neq \varnothing$ and $L \cap N \neq \varnothing$. 
  \end{enumerate}
 \end{enumerate}
\end{lem}

\begin{proof} 
By Assertion 1 of Proposition~\ref{propo Delta 234}, we have $o_T(M) \leq 1$ and $o_T(N) \leq 1$ because $M, N \in \text{mc}(T)$. Therefore $\widetilde{M}$ and $\widetilde{N}$ are well-defined (see Notation~\ref{Notation M}). Thus, the first assertion is satisfied.    
 
We now prove Assertion~2.1. Suppose toward a contradiction that $L$ and $\{x,y\}$ do not overlap.  
By Assertion~1 of Lemma~\ref{lem comodules}, $L$ is also a nontrivial module of $T$.
By Assertion~2 of Lemma~\ref{comod part}, and by interchanging $M$ and $N$, we may assume that $M$ is a module of $T$.

First suppose $L \cap \{x,y\} = \varnothing$. In this instance, $L$ is also a minimal co-module of $T$ (see Remark~\ref{rem mc overlap}). Therefore, if $M \cap L \neq \varnothing$, then $\widetilde{M} = M \cap L$ so that $x \in M \cap L$, contradicting $L \cap \{x,y\} = \varnothing$. Thus $M \cap L = \varnothing$. Similarly, $N \cap L = \varnothing$. It follows that $\{L,M,N\}$ is a co-modular decomposition of $T$, which contradicts $\Delta(T) = 2$. 

Second suppose $\{x,y\} \subseteq L$. By Assertion~4 of Proposition~\ref{propo modules}, $L \cup M$ is a module of $T$.
Recall that $N$ or $\overline{N}$ is a module of $T$. 
To begin, suppose that $\overline{N}$ is a module of $T$. By Assertion~4 of Proposition~\ref{propo modules}, $\overline{N \setminus L} = L \cup \overline{N}$ is a module of $T$. By minimality of the co-module $N$ of $T$, the module $\overline{N \setminus L}$ of $T$ is trivial. Thus $N \subseteq L$. It follows that the module $L \cup M$ of $T$ contains $M \cup N$. Therefore, $V(T) = L \cup M$ by Assertion~3 of Lemma~\ref{comod part}. Thus $\overline{M \setminus L} = L$. It follows that $\overline{M \setminus L}$ is a nontrivial module of $T$, which contradicts the minimality of the co-module $M$ of $T$.
Now suppose that $N$ is a module of $T$. In this instance, since $L$, $M$ and $N$ are modules of $T$, then $L \cup M$, $L \cup N$ and $L \cup M \cup N$ are also modules of $T$ by Assertion~4 of Proposition~\ref{propo modules}. It follows from Assertion~3 of Lemma~\ref{comod part} that the module $L \cup M \cup N$ of $T$ is trivial. Therefore $V(T) = L \cup M \cup N$. Thus $\overline{M \setminus L} = L \cup N$.  Since $\overline{M \setminus L} = L \cup N$ is a module of $T$, it follows from the minimality of the co-module $M$ of $T$ that the module $L \cup N$ of $T$ is trivial. Thus $V(T) = L \cup N$. Similarly, we have $V(T) = L \cup M$. Since $V(T) = L \cup M = L \cup N$ and $M \cap N = \varnothing$, then $V(T) = L$, a contradiction because $L$ is a nontrivial module of $T$. This completes the proof of Assertion~2.1.               
 
For the proof of Assertion~2.2, since $L$ and $\{x,y\}$ overlap by Assertion~2.1, then by interchanging $M$ and $N$, we may assume $L \cap \{x,y\} = \{y\}$. Thus, we only have to prove that $L \cap M \neq \varnothing$. Suppose not. We have $x \equiv_{T'} L$  because $L$ is a module of $T'$ and $x \notin L$. Since $L \cap \{x,y\} = \{y\}$ and $|L| \geq 2$, we obtain $x \not\equiv_{T} L$. In particular, $x \not\equiv_{T} \overline{M}$ because $L \subseteq \overline{M}$. Therefore, $\overline{M}$ is not a module of $T$ so that $M$ is a nontrivial module of $T$. 
Since $L$ and $M$ are disjoint nontrivial modules, there are distinct $z,t \in V(T) \setminus \{x,y\}$ such that $\{x,z\} \subseteq M$ and $\{y,t\} \subseteq L$. Since $T'(x,t) = T'(x,y)$  
because $L$ is a module of $T'$, then $T(x,t) \neq T(x,y)$. Moreover, $T(x,t) = T(z,t)$ and $T(x,y) = T(z,y)$ because $M$ is a module of $T$. It follows that $T(z,y) \neq T(z,t)$ and thus $T'(z,y) \neq T'(z,t)$, a contradiction because $L$ is a module of $T'$, $\{y,t\} \subseteq L$ and $z \notin L$. Thus $L \cap M \neq \varnothing$ as desired.     
\end{proof}

\begin{lem} \label{lem2 Del2}
 Let $T$ be a tournament such that $\Delta(T) = 2$ and satisfying the following hypothesis~(\ref{eq H}).
 \begin{equation} \label{eq H}
  \text{For every} \ x \in V(T), \ \text{the tournament} \ T-x \ \text{is decomposable}.
  \tag{$H$}
 \end{equation}
Consider a $\delta$-decomposition $\{M,N\}$ of $T$ such that $M$ is a nontrivial module of $T$. Recall that $\widetilde{M}$ and $\widetilde{N}$ are well-defined (see Assertion~1 of Lemma~\ref{lem Delta=2}). Let $x \in \widetilde{M}$ and $y \in \widetilde{N}$. If the tournament $T' = {\rm Inv}(T, \{x,y\})$ is decomposable, then the following three assertions are satisfied.
\begin{enumerate}
 \item $N$ is not a nontrivial module of $T$. In particular $\widetilde{N} = N$.
 \item Given a minimal co-module $L$ of $T'$, if $L$ is a nontrivial module of $T'$, then $L = \{x,z\}$ for some $z \in N \setminus \{y\}$.
 \item There exists a unique vertex $z \in N \setminus \{y\}$ such that $\{x,z\}$ is a twin of $T'$. 
\end{enumerate}
 \end{lem}

\begin{proof}
We easily verify that no tournament $U$ on at most four vertices satisfies both hypotheses~(\ref{eq H}) and $\Delta(U) = 2$.
Thus $v(T) \geq 5$. Suppose that the tournament $T' = \text{Inv}(T, \{x,y\})$ is decomposable. It follows from Remark~\ref{rem cherifa} that $T'$ admits a minimal co-module $L$ which is a nontrivial module of $T'$.

For the first assertion, suppose for a contradiction that $N$ is a nontrivial module of $T$. In this instance, we may interchange $M$ and $N$. Thus, by Assertion~2.1 of Lemma~\ref{lem Delta=2}, we can suppose $L \cap \{x,y\} = \{x\}$. Since $L \cap N \neq \varnothing$ by Assertion~2.2 of Lemma~\ref{lem Delta=2}, then $L \cup N$ is a module of $T$ by Assertion~3 of Lemma~\ref{lem comodules}. By Assertion~4 of Proposition~\ref{propo modules}, $L \cup N \cup M$ is a module of $T$. It follows from Assertion~3 of Lemma~\ref{comod part} that 
$L \cup N \cup M = V(T)$. Thus $\overline{M \setminus L} = L \cup N$. Since $M \setminus L$ is not a co-module of $T$ by minimality of the co-module $M$ of $T$, it follows that the module $L \cup N$ of $T$ is trivial. Therefore, $L \cup N = V(T)$ and thus $M \subseteq L$. On the other hand, since $y \equiv_T M$ and $|M| \geq 2$ because $M$ is a nontrivial module of $T$, then $y \not\equiv_{T'} M$. Since $M \subseteq L$, it follows that $y \not\equiv_{T'} L$, which contradicts that $L$ is a module of $T'$. 
 Thus, $N$ is not a nontrivial module of $T$, as desired. Therefore, $\widetilde{N} = N$ because $o_T(N) = 0$ by Lemma~\ref{degre G_T}.  

We now prove the second assertion. To begin, we prove that $L \cap \{x,y\} = \{x\}$. Suppose not. 
By Assertion~2.1 of Lemma~\ref{lem Delta=2}, we have $L \cap \{x,y\} = \{y\}$. Recall that $N$ is not a nontrivial module of $T$ by the first assertion of the lemma. Thus, $\overline{N}$ is a module of $T$. 
It follows from Assertions~2 and 1 of Lemma~\ref{lem comodules} that  $L \cap \overline{N}$ is a module of $T'$. More precisely, $L \cap \overline{N}$ is a trivial module of $T'$ because $L \cap \overline{N} \varsubsetneq L$ and $L$ is a minimal co-module of $T'$.   
Moreover, $L \cap M \neq \varnothing$ by Assertion~2.2 of Lemma~\ref{lem Delta=2}. Thus $|L \cap \overline{N}| = |L \cap M| =1$. It follows that $L \subseteq M \cup N$.   
 On the other hand, $\overline{N} \cup L$ is a module of $T$ by Assertion~3 of Lemma~\ref{lem comodules}. But $\overline{N} \cup L = \overline{N \setminus L}$. It follows from the minimality of the co-module $N$ of $T$ that the module $\overline{N \setminus L}$ of $T$ is trivial. 
  Therefore $N \setminus L = \varnothing$, i.e., $N \subseteq L$. Thus $L \cup M  \cup N = L \cup M$. Since $L \cup M$ is a module of $T$ by Assertion~3 of Lemma~\ref{lem comodules}, it follows that $V(T) = L \cup M$ by Assertion~3 of Lemma~\ref{comod part}. Recall that $L \subseteq M \cup N$. Thus
 \begin{equation} \label{abc}
  V(T) = L \cup M = M \cup N.
 \end{equation}
Notice that since $M$ is a minimal nontrivial module of $T$ (see Remark~\ref{fact1 min mod comod}), the tournament $T[M]$ is indecomposable (see Remark~\ref{fact2 min mod indec}). It follows that $N \neq \{y\}$, otherwise $T[M] = T-y$ because $V(T) = M \cup N$  (see (\ref{abc})), so that $T-y$ is indecomposable contrary to the hypothesis~(\ref{eq H}).
Suppose for a contradiction that $L \neq \overline{\{x\}}$. 
Since $V(T) \setminus (L \cup \{x\}) = M \setminus (L \cup \{x\})$ because $V(T) = M \cup L$ (see (\ref{abc})), then $M \setminus (L \cup \{x\}) \neq \varnothing$ because $L \neq \overline{\{x\}}$.
Pick $u \in  N \setminus \{y\}$ and $w \in M \setminus (L \cup \{x\})$. We have $T(w,y) = T(x,y)$ and $T(w,u) = T(x,u)$ because $M$ is a module of $T$. 
Moreover, since $L$ is a module of $T'$, $\{y,u\} \subseteq N \subseteq L$ (see (\ref{abc})) and $x \notin L$, then $T'(x,y) = T'(x,u)$ and thus $T(x,y) \neq T(x,u)$. It follows that $T(w,y) \neq T(w,u)$ and thus $T'(w,y) \neq T'(w,u)$, a contradiction because $L$ is a module of $T'$. Therefore $L = \overline{\{x\}}$. It follows from Remarks~\ref{fact1 min mod comod} and \ref{fact2 min mod indec} that $T'[L] = T'-x$ is indecomposable. Since $T'-x = T-x$, this again contradicts the hypothesis~(\ref{eq H}). We conclude that $L \cap \{x,y\} = \{x\}$ as claimed.   
 
 Since $L$ is a module of $T'$, we have $y \equiv_{T'} L$.
 Moreover, since $\overline{N}$ is a nontrivial module of $T$ by the first assertion of the lemma, then $y \equiv_{T} \overline{N}$ so that $y \not\equiv_{T'} \overline{N}$ and $y \equiv_{T'} \overline{N} \setminus \{x\}$. Therefore, if $\{x\} \varsubsetneq L \cap \overline{N}$, then $y \not\equiv_{T'} L \cap \overline{N}$, which contradicts $y \equiv_{T'} L$. Thus $L \cap \overline{N} = \{x\}$. On the other hand, since $y \notin L$, and $L \cap N \neq \varnothing$ by Assertion 2.2 of Lemma~\ref{lem Delta=2}, then $L \cap (N \setminus \{y\}) \neq \varnothing$. Let $z \in L \cap (N \setminus \{y\})$.
 Since $L \cap \overline{N} = \{x\}$, to show that $L = \{x,z\}$, which finishes the proof of the second assertion, it suffices to show that $L \cap N = \{z\}$. 
 By Assertion~1 of Proposition~\ref{propo modules}, $M$ and $L$ are modules of $T-y = T'-y$. Since $M \setminus L \neq \varnothing$ because $|M| \geq 2$ and $L \cap \overline{N} =\{x\}$, then $L \setminus M$ is a module of $T-y$ by Assertion~5 of Proposition~\ref{propo modules}. Moreover, since $y \equiv_{T'} L$ and thus $y \equiv_{T} L \setminus \{x\}$, then  $y \equiv_{T} L \setminus M$. It follows that $L \setminus M$ is a module of $T$. But $L \setminus M = L \cap N$ because $L \cap \overline{N} = \{x\} \subseteq M$. Since $L \cap N \varsubsetneq N$, 
 it follows from the minimality of the co-module $N$ of $T$ that the module $L \cap N$ of $T$ is trivial. Since $z \in N \cap L$ and $ N \cap L \varsubsetneq V(T)$, we obtain $N \cap L = \{z\}$ as claimed. 
 
 Lastly, we prove the third assertion. It follows from the second assertion of the lemma and from Remark~\ref{rem cherifa} that there exists $z \in N \setminus \{y\}$ such that $\{x,z\}$ is a twin of $T'$. Let $z' \in N \setminus \{y\}$ such that $\{x,z'\}$ is a twin of $T'$. We will prove that $\{z,z'\}$ is a module of $T$, which implies that $z=z'$ because $N$ is a minimal co-module of $T$. By Assertion~4 of Proposition~\ref{propo modules}, $\{x,z,z'\}$ is a module of $T'$ and thus of
 $T'-y=T-y$. Since $M$ and $\{x,z,z'\}$ are modules of $T-y$, and $M \setminus \{x,z,z'\} = M \setminus \{x\} \neq \varnothing$ because the module $M$ is nontrivial, then $\{x,z,z'\} \setminus M = \{z,z'\}$ is a module of $T-y$ by Assertion~5 of Proposition~\ref{propo modules}. Moreover, since $T'(y,z) = T'(y,z')$ because $\{x,z,z'\}$ is a module of $T'$, then $T(y,z) = T(y,z')$. It follows that $\{z,z'\}$ is a module of $T$, completing the proof.
 \end{proof}

The following lemma complements Lemma~\ref{lem2 Del2} when the hypothesis~(\ref{eq H}) is not satisfied. 

\begin{lem} [\cite{Index}] \label{lem fond}
Given a decomposable tournament $T$ such that $v(T)\geq 5$, if there exists $x \in V(T)$ such that $T-x$ is indecomposable, then $\delta(T)=1$. 
\end{lem} 

\begin{proof}[Proof of Proposition~\ref{prop Delta 2}]
Let $T$ be a tournament with at least five vertices.
If $\delta(T) =1$, then $\Delta(T) = 2$ by (\ref{eq dec delta}) and (\ref{eq deltaDelta}). Conversely, suppose $\Delta(T) =2$.
By Lemma~\ref{lem fond}, to prove that $\delta(T) =1$, we can suppose that for every $x \in V(T)$, the tournament $T-x$ is decomposable. Consider a $\delta$-decomposition $\{M,N\}$ of $T$. By Remark~\ref{rem cherifa} and by interchanging $M$ and $N$, we may assume that $M$ is a nontrivial module of $T$. By Assertion~1 of Lemma~\ref{lem Delta=2}, $\widetilde{M}$ and $\widetilde{N}$ are well-defined. Fix $x \in \widetilde{M}$. Suppose toward a contradiction that for every $y \in \widetilde{N}$, $\text{Inv}(T, \{x,y\})$ is decomposable. By Assertion~1 of Lemma~\ref{lem2 Del2}, we have $N = \widetilde{N}$. 
It follows from Assertion~3 of Lemma~\ref{lem2 Del2} that there exists a sequence $(z_n)_{n \in \mathbb{N}}$ of vertices of $N$ such that for every positive integer $n$, $\{x,z_n\}$ is a twin of $\text{Inv}(T,\{x,z_{n-1}\})$, and $z_n \neq z_{n-1}$. Since $N$ is finite and  $\{z_n: n \in \mathbb{N}\} \subseteq N$, the vertices $z_n$ are not pairwise distinct. Consider the smallest positive integer $k$ such that $z_k \in \{z_0, \ldots, z_{k-1}\}$. Since $z_k \neq z_{k-1}$, we have $k \geq 2$ and $z_k \in \{z_0, \ldots, z_{k-2}\}$. For every integer $n$, let $T_n = \text{Inv}(T,\{x,z_{n}\})$.  If $z_k \neq z_0$, i.e., $k\geq 3$ and $z_k = z_i$ for some  $i \in \{1, \ldots k-2\}$, then $\{x,z_k\}$ is a module of both $T_{k-1}$ and $T_{i-1}$, which is not possible because $z_{k-1} \neq z_{i-1}$. Thus $z_k = z_0$. To show that $\{z_0, \ldots, z_{k-1}\}$ is a module of $T$, we first consider a vertex $v \in V(T) \setminus (\{z_0, \ldots, z_{k-1}\} \cup \{x\})$. For every $j \in \{0, \ldots, k-1\}$, we have $T_{j}(v,z_{j+1}) = T_{j}(v,x)$ and thus $T(v,z_{j+1}) = T(v,x)$. It follows that $v \equiv_T \{z_1, \ldots, z_k\}$. Since $\{z_1, \ldots, z_{k}\} =\{z_0, \ldots, z_{k-1}\}$ because $z_k=z_0$, we obtain
   \begin{equation} \label{eq proof propo13}
   v \equiv_T \{z_0, \ldots, z_{k-1}\} \ \text{for every} \ v \in V(T) \setminus (\{z_0, \ldots, z_{k-1}\} \cup \{x\}). 
   \end{equation}
   Thus, $\{z_0, \ldots, z_{k-1}\}$ is a module of $T-x$. 
   Moreover, we have $x \equiv_T \{z_0, \ldots, z_{k-1}\}$ because $\overline{N}$ is a nontrivial module of $T$ (see Assertion~1 of Lemma~\ref{lem2 Del2}), $\{x\} \varsubsetneq \overline{N}$, and for every $w \in \overline{N} \setminus \{\overline{}x\}$, we have $w \equiv_T \{z_0, \ldots, z_{k-1}\}$ by (\ref{eq proof propo13}). It follows that $\{z_0, \ldots, z_{k-1}\}$ is a module of $T$. Moreover, the module $\{z_0, \ldots, z_{k-1}\}$ of $T$ is nontrivial because $k \geq 2$. Therefore, since $\{z_0, \ldots, z_{k-1}\} \subseteq N$ and $N$ is not a nontrivial module of $T$ (see Assertion~1 of Lemma~\ref{lem2 Del2}), we obtain $\{z_0, \ldots, z_{k-1}\} \varsubsetneq N$. This contradicts the minimality of the co-module $N$ of $T$. We conclude that there exists a vertex $y \in \widetilde{N}$ such that $\text{Inv}(T,\{x,y\})$ is indecomposable. Thus $\delta(T) =1$.          
\end{proof}

\begin{Discussion} \normalfont
Let $T$ be a tournament with at least five vertices such that $\Delta(T) = 2$, i.e., such that $\delta(T) = 1$ (see Proposition~\ref{prop Delta 2}). Suppose a $\delta$-decomposition $\{M,N\}$ of $T$ is given. We would like to discuss the following question. For which arcs $a$ of $T$, is the tournament $\text{Inv}(T,a)$ indecomposable?

Suppose that $\text{Inv}(T,a)$ is indecomposable. Recall that $o_T(M) \leq 1$ and $o_T(N) \leq 1$ (see Assertion~1 of Lemma~\ref{lem Delta=2}). Clearly, $\mathcal{V}(a) \cap M \neq \varnothing$ and $\mathcal{V}(a) \cap N \neq \varnothing$ (see Remark~\ref{rem mc overlap}). We will prove that $\mathcal{V}(a) \cap \widetilde{M} \neq \varnothing$ and $\mathcal{V}(a) \cap \widetilde{N} \neq \varnothing$. Suppose for a contradiction that $\mathcal{V}(a) \cap \widetilde{M} = \varnothing$. In this instance, $\widetilde{M} \neq M$. By Remark~\ref{rem1} and up to isomorphism, we can suppose that $\{0,1,2\}$ is a module of $T$, $T[\{0,1,2\}] = \underline{3}$, and $\widetilde{M} = \{1\}$. By Assertion~1 of Lemma~\ref{lem comodules}, since $1 \not\in \mathcal{V}(a)$ and $\{0,1\}$ (resp. $\{1,2\}$) is a module of $T$, then $0 \in \mathcal{V}(a)$ (resp. $2 \in \mathcal{V}(a)$). Thus $\mathcal{V}(a) = \{0,2\}$. Again by Assertion~1 of Lemma~\ref{lem comodules}, $\{0,1,2\}$ is a module of the indecomposable tournament $\text{Inv}(T,a)$, a contradiction. Thus $\mathcal{V}(a) \cap \widetilde{M} \neq \varnothing$. Similarly, we have $\mathcal{V}(a) \cap \widetilde{N} \neq \varnothing$.     

Conversely, let $x \in \widetilde{M}$ and $y \in \widetilde{N}$. Suppose that $M$ and $N$ are nontrivial modules of $T$. In this instance the hypothesis~(\ref{eq H}) of Lemma~\ref{lem2 Del2} is satisfied. It follows from Assertion~1 of Lemma~\ref{lem2 Del2} that $\text{Inv}(T, \{x,y\})$ is indecomposable. Here we would like to note that we often have $\widetilde{M} = M$ and $\widetilde{N} = N$ (see Remark~\ref{rem1}). For example, this is the case if each of the modules M and N contains at least three vertices. In this instance, $\text{Inv}(T, \{x,y\})$ is indecomposable for every $x \in M$ and $y \in N$. 

Lastly, recall that we can suppose that $M$ is a nontrivial module of $T$ (see Remark~\ref{rem cherifa}). Suppose that the hypothesis~(\ref{eq H}) is satisfied. By the proof of Proposition~\ref{prop Delta 2}, for every $x \in \widetilde{M}$, there exists $y \in \widetilde{N}$ such that $\text{Inv}(T, \{x,y\})$ is indecomposable. For some details about the case where the hypothesis~(\ref{eq H}) is not satisfied, see the proof of {\cite[Lemma 5.1]{Index}}.             
\end{Discussion}

\subsection{Proofs of Propositions \ref{prop Delta 3} and \ref{prop Delta 4}} \label{subsection proof2}

We need the following lemma.

\begin{lem} \label{lem config}
 Given a tournament $T$ such that $\Delta(T) \geq 3$, consider a co-modular decomposition $\{M,N,L\}$ of $T$ such that $M \in {\rm mc}(T)$. Suppose there are $x \in M$ and $y \in N$ such that $x \equiv_T L$, $y \equiv_T L$, and $T(x,L) \neq T(y,L)$. Consider the tournament $T'={\rm Inv}(T, \{x,y\})$, and let $I \in {\rm mc}(T')$. If $I \cap M \neq \varnothing$ and $I \cap N = \varnothing$, then $x\notin I$ and $I \in O_T(M)$.   
\end{lem}

\begin{proof}
Suppose $I \cap M \neq \varnothing$ and $I \cap N = \varnothing$. 
 If $x \notin I$, then $I \in \text{mc}(T)$ (see Remark~\ref{rem mc overlap}), and since $M \cap I \neq \varnothing$, we obtain $I \in O_T(M)$. Therefore, it suffices to prove that $x \notin I$. Suppose for a contradiction that $x \in I$. We distinguish the following two cases.
 
 First suppose that $I$ is a nontrivial module of $T'$. In this instance, $y \equiv_{T'} I$ and thus $y \not\equiv_T I$. Since $I \subseteq \overline{N}$, it follows that $\overline{N}$ is not a module of $T$. Thus, $N$ is a nontrivial module of $T$. Let $y' \in N \setminus \{y\}$ and $x' \in I \setminus \{x\}$. Since $I$ is a module of $T'$, we have $T'(y,x) = T'(y,x')$ and $T'(y',x) = T'(y',x')$. Moreover, since $N$ is a module of $T$, we have $T(y,x') = T(y',x')$ and thus $T'(y,x') = T'(y',x')$. It follows that $T'(y,x) = T'(y',x)$. Therefore $T(y,x) \neq T(y',x)$, contradicting that $N$ is a module of $T$. 
 
  Second suppose that $I$ is not a nontrivial module of $T'$. In this instance, $\overline{I}$ is a module of $T'$. On the other hand, since $L$ is a co-module of $T$, then by Assertion~1 of Lemma~\ref{lem comodules}, $L$ is also a co-module of $T'$. Moreover, $I \setminus L \neq \varnothing$ because $I \cap M \neq \varnothing$ and $M \cap L = \varnothing$. It follows from the minimality of the co-module $I$ of $T'$ that $L \setminus I \neq \varnothing$. Let $z \in L \setminus I$. We have $T'(x,y) = T'(x,z)$ because $\overline{I}$ is a module of $T'$. Thus $T(y,x) = T(x,z)$. Moreover, $T(x,z) \neq T(y,z)$ because $T(x,L) \neq T(y,L)$. It follows that $T(y,x) \neq T(y,z)$. Therefore, $\overline{N}$ is not a module of $T$. So $N$ is a nontrivial module of $T$. Let $y' \in N \setminus \{y\}$. Because $\overline{I}$ is a module of $T'$  and $N \cap I = \varnothing$, we have $T'(x,y) = T'(x,y')$ and thus $T(x,y) \neq T(x,y')$, which contradicts that $N$ is a module of $T$.                       
\end{proof}

\begin{proof}[Proof of Proposition \ref{prop Delta 3}]
 Let $T$ be a tournament with at least five vertices  such that $\Delta(T) = 3$. There exists a $\delta$-decomposition $\{M,N,L\}$ of $T$ such that $o_T(M) \leq 1$, $o_T(N) \leq 1$, $o_T(L) \leq 1$ (see Assertion~2 of Proposition~\ref{propo Delta 234}), and 
 \begin{equation} \label{eq proof3}
 \text{there are} \ x \in \widetilde{M}, \ y \in \widetilde{N}, \ \text{and} \ z \in \widetilde{L} \ \text{satisfying} \ T(x,z) = T(z,y) = 1. 
 \end{equation}
 Thereby, $\overline{L}$ is not a module of $T$. Thus, $L$ is a module of $T$. Therefore, it follows from (\ref{eq proof3}) that 
 \begin{equation} \label{eq N}
 T(x,L) = T(L,y) =1. 
 \end{equation}
 We consider the tournament $T'= \text{Inv}(T, \{x,y\})$. 
 If $\text{mc}(T')$ admits two disjoint elements $X$ and $Y$ such that $X \cap (M \cup N) = \varnothing$ and $Y \cap (M \cup N) = \varnothing$, then $\{X,Y, M,N\}$ is a co-modular decomposition of $T$ (see Remark~\ref{rem mc overlap}), which contradicts $\Delta(T) = 3$.
It follows that 
\begin{equation} \label{eq eq}
 \text{for every} \ \delta\text{-decomposition} \ D \ of \ T', \ |\{X \in D : X \cap (M \cup N) = \varnothing\}| \leq 1.
\end{equation}
We will prove that $\Delta(T') \leq 2$. Suppose not. By (\ref{eq eq}), there are disjoint $I, J \in \text{mc}(T')$ such that $I \cap (M \cup N) \neq \varnothing$ and $J \cap (M \cup N) \neq \varnothing$. Since $L \in \text{mc}(T)$ and $\{x,y\}\cap L = \varnothing$, then $L \in \text{mc}(T')$ (see Remark~\ref{rem mc overlap}). Since $I,J$ and $L$ are pairwise distinct elements of $\text{mc}(T')$, then by minimality of $I$ and $J$, 
\begin{equation} \label{eq NIJ}
L \setminus I \neq \varnothing \ \text{and} \ L \setminus J \neq \varnothing.    
\end{equation}
Suppose for a contradiction that 
\begin{equation} \label{eq eq eq}
I \cap M  \neq \varnothing, \ I \cap N \neq \varnothing, \ J \cap M \neq \varnothing, \ \text{and} \ J \cap  N \neq \varnothing.
\end{equation}
Suppose to the contrary that $M$ and $N$ are modules of $T$. By (\ref{eq N}), $T(L,M) =0 \neq T(L,N) =1$. Since $T'(L,M) = T(L,M)$ and $T'(L,N) = T(L,N)$, we obtain $T'(L,M) = 0 \neq T'(L,N) =1$. Therefore, it follows from (\ref{eq NIJ}) and (\ref{eq eq eq}) that neither $I$ nor $J$ is a module of $T'$, which contradicts Assertion~2 of Lemma~\ref{comod part}. Thus, $M$ or $N$ is not a module of $T$. By interchanging $T$ and $T^{\star}$, as well as $M$ and $N$, we may assume that $N$ is not a module of $T$. By Lemma~\ref{degre G_T}, $o_T(N)=0$. It follows that $I \notin \text{mc}(T)$ and $J \notin \text{mc}(T)$, otherwise $I$ or $J$ belongs to $O_T(N)$, which contradicts $o_T(N)=0$. Thus, $I \cap \{x,y\} \neq \varnothing$ and $J \cap \{x,y\} \neq \varnothing$ (see Remark~\ref{rem mc overlap}). Therefore, since $I \cap J = \varnothing$, then by interchanging $I$ and $J$, we may assume $I \cap \{x,y\} = \{x\}$ and $J \cap \{x,y\} = \{y\}$. Since $N$ is not a module of $T$, then by Assertion~2 of Lemma~\ref{comod part}, $M$ is a module of $T$. Thus by (\ref{eq N}) we get $T(L,M) = 0 \neq T(L,y) = 1$ and hence $T'(L,M) = 0 \neq T'(L,y) = 1$. Since $y \in J$, $M \cap J \neq \varnothing$ (see (\ref{eq eq eq})), and $L \setminus J \neq \varnothing$ (see (\ref{eq NIJ})), it follows that $J$ is not a module of $T'$ and hence $\overline{J}$ is a module of $T'$. Thus $T'(y,x) = T'(y, L \setminus J)$ and hence $T(y,x) \neq T(y, L \setminus J)$, contradicting that $\overline{N}$ is a module of $T$ because $N$ is not a module of $T$. We conclude that (\ref{eq eq eq}) does not hold. 

By interchanging $I$ and $J$, we may assume that $I \cap M = \varnothing$ or $I \cap N = \varnothing$. 
By interchanging $T$ and $T^{\star}$, as well as $M$ and $N$, we may assume $I \cap N = \varnothing$. Moreover, $I \cap M \neq \varnothing$ because $I \cap (M \cup N) \neq \varnothing$.  Thus, Lemma~\ref{lem config} applies to the $\delta$-decomposition $\{M,N,L\}$ of $T$, and yields $x \notin I$ and $I \in O_T(M)$. More precisely,  $O_T(M) = \{I\}$ because $o_T(M) \leq 1$ and $I \in O_T(M)$. Therefore $\widetilde{M} = M \cap I$, and since $x \in \widetilde{M}$, we get $\widetilde{M} = M \cap I = \{x\}$ (see Notation~\ref{Notation M}), contradicting $x \notin I$. We conclude that $\Delta(T') \leq 2$. Since $T'$ is decomposable because $\delta(T) \geq \left\lceil \frac{\Delta(T)}{2} \right\rceil =2$ (see (\ref{eq deltaDelta})), then $\Delta(T') \geq 2$ (see (\ref{eq dec delta})). Thus $\Delta(T')=2$. By Proposition~\ref{prop Delta 2}, $\delta(T')=1$. Since $\delta(T) \leq 1 + \delta(T')$, we obtain $\delta(T) \leq 2$. But since $\Delta(T)=3$, we have $\delta(T) \geq 2$ (see (\ref{eq deltaDelta})). Thus $\delta(T) = 2$, completing the proof.                        
\end{proof}

\begin{proof}[Proof of Proposition \ref{prop Delta 4}]
Let $T$ be a tournament such that $\Delta(T) \geq 4$.
By Assertion~3 of Proposition~\ref{propo Delta 234}, $T$ admits a $\delta$-decomposition $D$ which contains four elements $M_1, M_2, M_3$ and $M_4$ satisfying the following conditions.
  \begin{enumerate}
  \item[(C1)] For every $i \in \{1,3,4\}$, $o_T(M_i) \leq 1$.
  \item[(C2)] $T(M_1,M_2) = T(M_2,M_3) = 1$.
  \item[(C3)] There exists $u \in M_4$ such that $T(u,M_1) =1$ or $T(M_3,u) =1$.
 \end{enumerate}
 By (C1), $\widetilde{M_1}$ and $\widetilde{M_3}$ are well-defined (see Notation~\ref{Notation M}). Recall that $\widetilde{M_1} \neq \varnothing$ and  $\widetilde{M_3} \neq \varnothing$. Let $x \in \widetilde{M_1}$ and $y \in \widetilde{M_3}$. Consider the tournament $T' = \text{Inv}(T, \{x,y\})$. We claim that 
 \begin{equation} \label{eq Delta4}
  \text{for every} \ I \in {\rm mc}(T'), \ \text{we have} \ I \cap (M_1 \cup M_3) = \varnothing.  
 \end{equation}
Before proving (\ref{eq Delta4}), we first use it to show that $\Delta(T') = \Delta(T) - 2$, which completes the proof. By (\ref{eq Delta4}) and by Remark~\ref{rem mc overlap}, we have 
\begin{equation} \label{eq mnc TT'}
\text{mc}(T') \subseteq \text{mc}(T) \setminus \{M_1,M_3\}.    
\end{equation}
Let $D'$ be a $\delta$-decomposition of $T'$. It follows from (\ref{eq mnc TT'}) and (\ref{eq Delta4}) that $D' \cup \{M_1,M_3\}$ is a co-modular decomposition of $T$. Thus 
$|D' \cup \{M_1,M_3\}| = |D'|+2 \leq \Delta(T)$. Since $|D'| = \Delta(T')$, we obtain $\Delta(T') \leq \Delta(T)-2$. On the other hand, $D \setminus \{M_1,M_3\}$ is a co-modular decomposition of $T'$ (see Remark~\ref{rem mc overlap}). Since $|D \setminus \{M_1,M_3\}| = |D|-2 = \Delta(T) -2$, it follows that $\Delta(T') \geq \Delta(T)-2$. We conclude that $\Delta(T') = \Delta(T)-2$ as desired.  

We now prove (\ref{eq Delta4}). Suppose toward a contradiction that there exists $I \in \text{mc}(T')$ such that $I \cap (M_1 \cup M_3) \neq \varnothing$. We first show that $I \cap M_1 \neq \varnothing$ and $I \cap M_3 \neq \varnothing$. Suppose not. By interchanging $T$ and $T^{\star}$, as well as $M_1$ and $M_3$, we may assume $I \cap M_3 = \varnothing$. Moreover, $I \cap M_1 \neq \varnothing$ because $I \cap (M_1 \cup M_3) \neq \varnothing$. Thus, Lemma~\ref{lem config} applies to the co-modular decomposition $\{M,N,L\}$ of $T$, where $(M,N,L) = (M_1,M_3,M_2)$, and yields $x \notin I$ and $I \in O_{T}(M_1)$. Since $o_T(M_1) \leq 1$ (see (C1)), we obtain $O_T(M_1) = \{I\}$. Therefore, $\widetilde{M_1} = M_1 \cap I$ (see Notation~\ref{Notation M}), which is a contradiction because $x \in \widetilde{M_1}$ and $x \notin I$. Thus $I \cap M_1 \neq \varnothing$ and $I \cap M_3 \neq \varnothing$. We now show that $I$ is not a module of $T'$. Since $I$ and $M_2$ are distinct elements of $\text{mc}(T')$ (see Remark~\ref{rem mc overlap}), then by minimality of $I$ we have $M_2 \setminus I \neq \varnothing$. Moreover, we have $T(M_2 \setminus I, M_1 \cap I) =0 \neq T(M_2 \setminus I, M_3 \cap I) =1$ (see (C2)) and thus $T'(M_2 \setminus I, M_1 \cap I) =0  \neq T'(M_2 \setminus I, M_3 \cap I) =1$. Therefore, $I$ is not a module of $T'$. So $\overline{I}$ is a module of $T'$. Moreover, since $o_{T'}(I) =0$ by Lemma~\ref{degre G_T}, and $M_2, M_4 \in \text{mc}(T')$ by Remark~\ref{rem mc overlap}, then $I \cap M_2 = I \cap M_4 = \varnothing$.
Let $v \in I \cap M_1$ and $w \in I \cap M_3$. We have $T'(v, \overline{I})=1$ because $\overline{I}$ is a module of $T'$, $M_2 \subseteq \overline{I}$, and $T'(v,M_2) = T(v,M_2) =1$ (see (C2)). Similarly, we have $T'(w, \overline{I})=0$. Thus $T'(v, \overline{I}) = T'(\overline{I},w) =1$. Since $M_4 \subseteq \overline{I}$, it follows that $T'(v, M_4) = T'(M_4, w) = 1$ and thus $T(v, M_4) = T(M_4, w) = 1$, which contradicts (C3). This completes the proof.
\end{proof}


{}

\end{document}